\documentclass[10pt]{amsart}
\usepackage[margin=3cm]{geometry}
\usepackage{amsfonts, amssymb, amsmath}
\usepackage{stmaryrd}
\usepackage{graphics, graphicx}
\usepackage{color}\definecolor{Green}{rgb}{0.0, 0.5, 0.0}
\usepackage{hyperref}
\makeatletter
\renewcommand*{\eqref}[1]{%
  \hyperref[{#1}]{\textup{\tagform@{\ref*{#1}}}}%
}
\makeatother

\newtheorem{thm}{Theorem}[section]
\newtheorem{cor}[thm]{Corollary}
\newtheorem{prop}[thm]{Proposition}
\newtheorem{lem}[thm]{Lemma}
\newtheorem{claim}[thm]{Claim}
\theoremstyle{definition}
\newtheorem{defn}[thm]{Definition}
\newtheorem{ex}[thm]{Example}

\newtheorem{convention}[thm]{Convention}
\newtheorem{question}[thm]{Question}

\newtheorem*{ack}{Acknowledgments}

\numberwithin{equation}{section}

\newcommand{\Z}{\mathbb Z}
\newcommand{\N}{\mathbb N}
\newcommand{\kk}{\mathbf k}
\newcommand{\Ss}{\mathbb S}
\newcommand{\JJ}{\mathcal J_0}
\newcommand{\sgn}{\text{sgn}}
\newcommand{\In}{\text{In}}
\newcommand{\Out}{\text{Out}}
\newcommand{\Ve}{\text{Vert}}
\newcommand{\Ed}{\text{Edge}}
\newcommand{\Le}{\text{Leaf}}
\newcommand{\Ro}{\text{Root}}
\newcommand{\Int}{\text{Intern}}
\newcommand{\Det}{\text{Det}}
\newcommand{\mc}{\mathcal }
\newcommand{\bD}{{\bf D}}
\newcommand{\bC}{{\bf C}}
\newcommand{\ii}{{\,|\,}}
\newcommand{\oo}{\bigcirc}
\newcommand{\iii}{{\mathbf i}}
\newcommand{\ooo}{{\mathbf o}}

\title{Koszuality of the $\mc V^{(d)}$ dioperad}

\author[K.~Poirier]{Kate~Poirier}
  \address{Kate Poirier,
  Department of Mathematics, New York City College of Technology, City University of New York, 300 Jay Street, Brooklyn, NY 11201}
  \email{kpoirier@citytech.cuny.edu}

\author[T.~Tradler]{Thomas~Tradler}
  \address{Thomas Tradler,
  Department of Mathematics, New York City College of Technology, City University of New York, 300 Jay Street, Brooklyn, NY 11201}
  \email{ttradler@citytech.cuny.edu}

\date{\today}

\begin{document}
\maketitle

\begin{abstract}
Define a $\mc V^{(d)}$-algebra as an associative algebra with a symmetric and invariant co-inner product of degree $d$. Here, we consider $\mc V^{(d)}$ as a dioperad which includes operations with zero inputs. We show that the quadratic dual of $\mc V^{(d)}$ is $(\mc V^{(d)})^!=\mc V^{(-d)}$ and prove that $\mc V^{(d)}$ is Koszul. We also show that the corresponding properad is not Koszul contractible.
\end{abstract}

\setcounter{tocdepth}{2}
\tableofcontents

\section{Introduction}\label{SEC:Introduction}

In this paper, we study the dioperad $\mc V^{(d)}$. Informally, we may describe $\mc V^{(d)}$ by saying that a $\mc V^{(d)}$-algebra is given by an associative algebra $(A,\cdot )$ together with a symmetric and invariant element $c\in A\otimes A$ of degree $d$. Writing $c=\sum_i c'_{i}\otimes c''_{i}$, the symmetry and invariance conditions are expressed as 
\begin{align}\label{EQN:symmetry-co-ip}
 &\sum_i c'_{i}\otimes c''_{i}= \sum_i (-1)^{|c'_i||c''_i|}\cdot c''_i\otimes c'_i, \quad \text{ and }\\
 &\sum_i (a\cdot c'_{i})\otimes c''_{i}=\sum_i (-1)^{|a|(|c'_{i}|+|c''_{i}|)} \cdot c'_{i}\otimes (c''_{i}\cdot a) \text{ for all }a\in A,
\end{align}
respectively. 

An important example that we have in mind is $A=H^\bullet(X)$, the cohomology ring of an even-dimensional, oriented manifold $X$, with coefficients in a field $\kk$ of characteristic $0$.
Here, the element $c$ in $H^\bullet(X) \otimes H^\bullet(X)$ is obtained from the diagonal $\Delta: X \to X\times X$ as the image of the Thom class in $H^\bullet (X \times X, X \times X - \Delta(X))$ under the composition of restriction $H^\bullet(X \times X, X \times X - \Delta(X)) \to H^\bullet(X \times X)$ and the K\"unneth isomorphism $H^\bullet(X \times X) \stackrel{\sim}{\to} H^\bullet(X) \otimes H^\bullet(X)$.
Note that if $X$ is of dimension $d$, then the element $c$ satisfies $\sum_i c'_{i}\otimes c''_{i}=(-1)^d\cdot \sum_i (-1)^{|c'_i||c''_i|}\cdot c''_i\otimes c'_i$ (see \cite[page 128, Theorems 11.10 and 11.11]{MS}). Thus, if $d$ is even, $c$ satisfies the above relation \eqref{EQN:symmetry-co-ip}. Due to this example, we will also refer to the element $c$ as a {\em co-inner product}.
Note also that the cohomology class $\sum_i c'_i\cdot c''_i$ in $H^d(X)$ is equal to the Euler class of the tangent bundle of $X$.

For concreteness, we follow Gan's definitions and explicit constructions \cite{G}. Strictly speaking, $\mc V^{(d)}$ is not a dioperad as defined by Gan: in \cite{G} algebras over dioperads only have operations $A^{\otimes n}\to A^{\otimes m}$ with $n\ge 1$ inputs and $m\ge 1$ outputs, but $\mc V^{(d)}$ has as a generating operation the co-inner product, which is a $0$-to-$2$ operation. The more general definition of, for example, \cite{KW:FC} covers the case of $\mc V^{(d)}$ but in order to make things explicit and elementary we have chosen to modify Gan's definitions to include zero outputs.
Much of the content of Section \ref{SEC:dioperd0} (defining the cobar dual dioperad $\bD \mc P$ of a dioperad $\mc P$ and stating some of its basic properties) follows along Gan's discussion \cite[Sections 1-3]{G} but putting special emphasis to the case when $0$ inputs are allowed, especially when considering Koszulity of such a dioperad. One technical difference here is to use a convenient shift in the cobar dual dioperad, which allows generators to be of the form $E(1,2)$ (i.e. having $2$ inputs and $1$ output) and $E(2,0)$ (i.e. having $0$ inputs and $2$ outputs), while the relations live in $R(1,3)$ and $R(2,1)$.

In Section \ref{SEC:Vinfty}, we define $\mc V^{(d)}$ as the quadratic dioperad with generators in $E(1,2)$ (i.e. the product) and $E(2,0)$ (i.e. the co-inner product in degree $d$), and with relations in the free dioperad in degrees $\mc F_0(E)(1,3)$ (i.e. associativity) and $\mc F_0(E)(2,1)$ (i.e. invariance). A calculation then shows that $\mc V^{(d)}$ is closely related to its own quadratic dual. This is our first result.
\begin{prop}
The quadratic dual of $\mc V^{(d)}$ is $(\mc V^{(d)})^!=\mc V^{(-d)}$.
\end{prop}
The main result of this paper is that $\mc V^{(d)}$ is Koszul as a dioperad.
\begin{thm}\label{THM:V-Koszul}
$\mc V^{(d)}$ is Koszul; i.e. the map $\bD \mc V^{(-d)}\to \mc V^{(d)}$ is a quasi-isomorphism.
\end{thm}
The proof of Theorem \ref{THM:V-Koszul} relies on a result of our previous paper \cite{PT}, where we showed that a certain combinatorial cell complex (the ``assocoipahedron'' $Z_\alpha$) is contractible. The proof of Theorem \ref{THM:V-Koszul} follows from the observation that the dg-chain complex obtained by $\bD\mc V^{(-d)}$ is isomorphic to the cellular complex on a disjoint unions of assocoipahedra, which is thus contractible in each connected component.

Due to this result, we can define the dioperad $\mc V^{(d)}_\infty:=\bD (\mc V^{(d)})^!=\bD \mc V^{(-d)}$. Algebras over $\mc V^{(d)}_{\infty}$ are closely related to $V_\infty$-algebras that appeared in \cite{TZ} in the context of string topology type operations. More precisely, when $d$ is even, the concept of a $\mc V^{(d)}_\infty$-algebras coincides with that of a $V_\infty$-algebra, and, when $d$ is odd, the two concepts differ at most by signs. Moreover, it was shown in \cite{TZ} that there is a PROP $\mc{DG}^\bullet_\infty$ consisting of graph complexes of directed graphs, which has an action on the cyclic Hochschild chain complex $CC^\bullet(A)$ of a $V_\infty$-algebra $A$. Thus, we get the following corollary.
\begin{cor}[\cite{TZ}, Theorem 4.3]
If $d$ is even, and $A$ is a $\mc V^{(d)}_\infty$ algebra, then there is a map of PROPs $\mc{DG}^\bullet_\infty\to \mc End(CC^\bullet(A))$.
\end{cor}
Due to this result, $\mc V^{(d)}_\infty$ algebras are a natural context in which one might consider chain level string topology type operations for even dimensional manifolds. In fact, one of our goals for our future work is to give a precise comparison with the chain level string topology operations defined in \cite{DPR}.

In Section \ref{SEC:related-structures} we study some concepts related to $\mc V^{(d)}$. In Section \ref{SEC:V-properad}, we consider $\mc V^{(d)}$ as a properad, i.e. we look at the properad $\mc F_{dioperad}^{properad}\mc V^{(d)}$, which is the universal enveloping properad (as defined by Merkulov and Vallette in \cite[Section 5.6.]{MV}) applied to $\mc V^{(d)}$. We check that $\mc F_{dioperad}^{properad}\mc V^{(d)}$ is not contractible. However, we do not know if $\mc F_{dioperad}^{properad}\mc V^{(d)}$ is Koszul as a properad.

In Section \ref{SEC:W-dioperad}, we consider anti-symmetric co-inner products, such as they appear as the Thom class of an odd dimensional, oriented manifold. The corresponding dioperad $\mc W^{(d)}$ is a quotient of $\mc V^{(d)}$. At this point, we do not know whether $\mc W^{(d)}$ is a Koszul dioperad or not.

\begin{ack}
We would like to thank Gabriel Drummond-Cole for useful conversations concerning the topics of this paper.
\end{ack}

\section{Dioperads and Koszul Duality revisited}\label{SEC:dioperd0}

In \cite{G}, Gan defined dioperads and studied Koszulity of dioperads. This notion only allowed for $n\geq 1$ inputs and $m\geq 1$ outputs for each of its space of operations. Since our dioperad $\mc V^{(d)}$ encodes (among others) a co-inner product, which has $0$ inputs and $2$ outputs, we need dioperads to also allow for $0$ inputs. In this section, we describe these dioperads, which allow for such operations. Most of this section is just a rewrite of Gan's \cite[Sections 1-3]{G}. For this reason, we won't give a complete mirror of all the facts stated in \cite{G}, but merely give a minimal self-contained description, with a particular focus to the case of $0$ inputs.

\subsection{Dioperads}

\begin{convention}\label{CONVENTIONS}
In this paper, we work over a field $\kk$ of characteristic $0$. Unless stated otherwise, differential graded (or dg-)vector spaces $(V,d)$ consist of spaces $V=\{V^i\}_{i\in \Z}$ where the differential is of degree $+1$, i.e. $d^i:V^i\to V^{i+1}$. We denote by $V[k]$ the dg-vector space with $(V[k])^i:=V^{i+k}$, i.e. the dg-vector space that shifts $V$ \emph{down} by $k$. Furthermore, we apply the usual Koszul rules for signs, in particular, the dual dg-vector space $(V^*,d^*)$ is given by $(V^*)^i:=(V^{-i})^*=Hom(V^{-i},\kk)$ with $(d^*)^{i}(f)=(-1)^{|f|} f\circ d^{-i-1}$.

Let $\N=\{1,2,3,\dots\}$ be the natural numbers and $\N_0=\N\cup \{0\}$. Let $\Ss_n$ be the symmetric group on $n\in \N_0$ elements; in particular, for $n=0$, $\Ss_0=\{\phi\}$ is a one element set with the unique map $\phi:\varnothing\to \varnothing$ from the empty set to itself. Recall also the block permutation (\cite[Definition 1.2]{MSS}): for $m=m_1+\dots+m_n$ with $m_1,\dots, m_n\in \N$, and $\sigma \in \Ss_n$, the block permutation $\sigma_{m_1,\dots, m_n}\in \Ss_m$ permutes blocks of size $m_1,\dots, m_n$ just as $\sigma$ permutes $1,\dots, n$.

 For $\sigma \in \Ss_n$, denote by $\sgn(\sigma)\in \{+1,-1\}$ the sign of the permutation, where $\sgn(\phi)=+1$ for $\phi\in \Ss_0$. If $\sigma_1\in \Ss_{n_1}, \sigma_2\in \Ss_{n_2}$ with $n_1,n_2\in \N$, and $i\in \{1,\dots, n_1\}$, define $\sigma_1\circ_i \sigma_2\in \Ss_{n_1+n_2-1}$ to be
\[
\sigma_1\circ_i \sigma_2:=(\sigma_1)_{1,\dots,1, n_2,1,\dots, 1}\circ (id\times \dots\times id\times \sigma_2\times id\times \dots\times id),
\]
where $(\sigma_1)_{1,\dots,1, n_2,1,\dots, 1}$ is the block permutation. Furthermore, for $\phi\in \Ss_0$, define $\sigma_1\circ_i \phi\in \Ss_{n_1-1}$ to be
\[
\sigma_1\circ_i \phi := \bar\rho_{\sigma_1(i)} \circ \sigma_1\circ \rho_i,
\]
where $\rho_j:\{1,\dots, n_1-1\}\to \{1,\dots, n_1\}, \rho_j(k)=\left\{\begin{array}{ll}k,&\text{for }k< j\\ k+1, &\text{for }k\geq j\end{array}\right.$, and $\bar\rho_j:\{1,\dots, n_1\}\to \{1,\dots, n_1-1\}, \bar\rho_j(k)=\left\{\begin{array}{ll}k, &\text{for }k< j\\ k-1, &\text{for }k\geq j\end{array}\right.$. Note, that for all $\sigma_1\in \Ss_{n_1}, \sigma_2\in \Ss_{n_2}$ with $n_1\in \N$, $n_2\in \N_0$, we have
\[
\sgn(\sigma_1\circ_i\sigma_2)=\sgn(\sigma_1)\sgn(\sigma_2)(-1)^{(i-\sigma_1(i))(n_2-1)}.
\]
\end{convention}

We now recall the concept of a dioperad.
\begin{defn}\label{DEF:dioperad0}
Let $\JJ$ be the indexing set for the outputs and inputs of our space of operations; more precisely, we set $\JJ=\N\times \N_0-\{(1,0),(1,1)\}$.

A dioperad consists of 
\begin{itemize}
\item
dg-vector spaces $\mc P(m,n)$ for each $(m,n)\in \JJ$ which are $(\Ss_m,\Ss_n)$-bimodules,
\item
dg-morphisms $_i\circ _j:\mc P(m_1,n_1)\otimes \mc P(m_2,n_2)\to \mc P(m_1+m_2-1,n_1+n_2-1)$, called compositions, of degree $0$ for each $(m_1,n_1), (m_2,n_2)\in \JJ$ with $n_1\neq 0$ and $1\leq i \leq n_1, 1\leq j \leq m_2$,
\end{itemize}
subject to the conditions (a), (b), and (c) below. (For ease of presentation, we do not require a unit morphism, so that a more accurate name for this structure might have been a ``pseudo-dioperad.'')

We require the following (cf. \cite[1.1]{G}):
\begin{enumerate}
\item[(a)] for $(m_1,n_1),(m_2,n_2), (m_3,n_3)\in \JJ$ with $n_1\neq 0$, $n_1+n_2-1\neq 0$, $1\leq i\leq n_1+n_2-1$, $1\leq j\leq m_3$, $1\leq k\leq n_1$, $1\leq l\leq m_2$, there is an equality
\[
_i\circ_j(_k\circ_l\otimes id)=\left\{
\begin{array}{ll}
(\sigma,1)(_{k+n_3-1}\circ_{l})(_i\circ_j\otimes id)(id\otimes \tau) & \text{for }i\leq k-1 \\ 
_{k}\circ_{j+l-1}(id\otimes _{i-k+1}\circ_{j}) & \text{for }k\leq i\leq k+n_2-1 \\ 
(\sigma,1)(_{k}\circ_{l})(_{i-n_2+1}\circ_{j}\otimes id)(id\otimes \tau) & \text{for }k+n_2\leq i,
\end{array}\right.
\]
as morphisms $\mc P(m_1,n_1)\otimes \mc P(m_2, n_2)\otimes \mc P(m_3, n_3)\to \mc P(m_1+m_2+m_3-2,n_1+n_2+n_3-2)$, where $\tau:\mc P(m_2,n_2)\otimes \mc P(m_3, n_3)\to P(m_3,n_3)\otimes \mc P(m_2, n_2)$ is the symmetry isomorphism, and $\sigma\in \Ss_{m_1+m_2+m_3-2}$ is the block permutation $\sigma=((12)(45))_{l-1,j-1,m_1,m_3-j,m_2-l}$.
\item[(b)] for $(m_1,n_1),(m_2,n_2), (m_3,n_3)\in \JJ$ with $n_1\neq 0$ and $n_2\neq 0$, $1\leq i\leq n_1$, $1\leq j\leq m_2+m_3-1$, $1\leq k\leq n_2$, $1\leq l\leq m_3$, there is an equality
\[
_i\circ_j(id\otimes_k\circ_l)=\left\{
\begin{array}{ll}
(1,\sigma)(_{k}\circ_{l+m_1-1})(id\otimes _i\circ_j)(\tau\otimes id) & \text{for }j\leq l-1 \\ 
_{k+i-1}\circ_{l}(_{i}\circ_{j-l+1}\otimes id) & \text{for }l\leq j\leq l+m_2-1 \\ 
(1,\sigma)(_{k}\circ_{l})(id\otimes _{i}\circ_{j-m_2+1})(\tau\otimes id) & \text{for }l+m_2\leq j,
\end{array}\right.
\]
as morphisms $\mc P(m_1,n_1)\otimes \mc P(m_2, n_2)\otimes \mc P(m_3, n_3)\to \mc P(m_1+m_2+m_3-2,n_1+n_2+n_3-2)$, where $\tau:\mc P(m_1,n_1)\otimes \mc P(m_2, n_2)\to \mc P(m_2,n_2)\otimes \mc P(m_1, n_1)$ is the symmetry isomorphism, and $\sigma\in \Ss_{n_1+n_2+n_3-2}$ is the block permutation $\sigma=\left\{\begin{array}{ll}((12)(45))_{i-1,k-1,n_3,n_2-k,n_1-i} & \text{for }n_3>0 \\ ((12)(34))_{i-1,k-1,n_2-k,n_1-i}  & \text{for }n_3=0 \end{array}\right.$.
\item[(c)] for $(m_1,n_1),(m_2,n_2)\in \JJ$ with $n_1\neq 0$, $1\leq i\leq n_1$, $1\leq j\leq m_2$, $\pi_1\in \Ss_{m_1}$, $\sigma_1\in \Ss_{n_1}$, $\pi_2\in \Ss_{m_2}$, $\sigma_2\in \Ss_{n_2}$, there is an equality 
\[
_i\circ _j ((\pi_1,\sigma_1)\otimes (\pi_2,\sigma_2))=(\pi_2\circ_{\pi_2^{-1}(j)} \pi_1, \sigma_1\circ_i \sigma_2)(_{\sigma_1(i)}\circ_{\pi_2^{-1}(j)})
\]
as morphisms $\mc P(m_1,n_1)\otimes \mc P(m_2, n_2)\to \mc P(m_1+m_2-1,n_1+n_2-1)$.
\end{enumerate}

A morphism of dioperads $f:\mc P\to \mc Q$ consists of maps $f(m,n):\mc P(m,n)\to \mc Q(m,n)$ for all $(m,n)\in \JJ$ compatible with all $_i\circ_j$-compositions and $(\Ss_m,\Ss_n)$-actions. A morphism $f$ is a quasi-isomorphism if each $f(m,n)$ is a quasi-isomorphism.
\end{defn}

Note, that every dioperad $\mc P$ gives a dioperad in the sense of \cite{G} by forgetting the spaces $\mc P(m,0)$.
Conversely, every dioperad $\mc P$ in the sense of \cite{G} (in dg-vector spaces) gives a dioperad by setting $\mc P(m,0)=\{0\}$.

\begin{ex}\label{EX:End-dioperad0}
For a dg-vector space $V$, we define the endomorphism dioperad by $\mc End_V(m,n)=Hom(V^{\otimes n}, V^{\otimes m})$, with composition as in \cite[1.2]{G}: for $f\in\mc End(m_1,n_1), g\in \mc End(m_2,n_2)$ with $n_1\neq 0$, set
\[
f_i\circ_j g=(id\otimes\dots\otimes f\otimes \dots\otimes id)\sigma(id\otimes\dots\otimes g\otimes \dots\otimes id)
\]
with $\sigma=((12)(45))_{i-1,j-1,1,m_2-j,n_1-i}\in \Ss_{n_1+m_2-1}$.
\end{ex}
\begin{defn}
For a dioperad $\mc P$, a $\mc P$-algebra is a dg-vector space $A$, together with a morphism of dioperads $\mc P\to \mc End_A$.
\end{defn}
There are also various ``sign'' dioperads which are of interest.
\begin{ex}\label{EX:sign-dioperad0s}\quad 
\begin{enumerate}
\item
There is a dioperad $\Theta$ given by defining $\Theta(m,n)=\kk[1-n]$ to be a $1$-dimensional space concentrated in degree $n-1$. The $(\Ss_m,\Ss_n)$-action on $\Theta(m,n)$ is given by $(\pi,\sigma).1=\sgn(\sigma)$. The composition $_i\circ _j:\Theta(m_1,n_1)\otimes \Theta(m_2,n_2)\to \Theta(m_1+m_2-1,n_1+n_2-1)$, $1_i\circ _j 1=(-1)^{(i-1)(n_2-1)}$ is of degree $0$ and satisfies the required conditions of a dioperad.\\ 
Dually, $\Theta^{-1}(m,n)=\kk[n-1]$ is the dioperad in degree $1-n$ with a similar action $(\pi,\sigma).1=\sgn(\sigma)$ and composition $1_i\circ _j 1=(-1)^{(i-1)(n_2-1)}$.
\item
There is a dioperad $\Gamma$ given by defining $\Gamma(m,n)=\kk[1-m]$ to be a $1$-dimensional space concentrated in degree $m-1$. The $(\Ss_m,\Ss_n)$-action on $\Gamma(m,n)$ is given by $(\pi,\sigma).1=\sgn(\pi)$. The composition $_i\circ _j:\Gamma(m_1,n_1)\otimes \Gamma(m_2,n_2)\to \Gamma(m_1+m_2-1,n_1+n_2-1)$, $1_i\circ _j 1=(-1)^{(j-1)(m_1-1)}$ is of degree $0$ and satisfies the required conditions of a dioperad.\\
Dually, $\Gamma^{-1}(m,n)=\kk[m-1]$ is the dioperad in degree $1-m$ with a similar action $(\pi,\sigma).1=\sgn(\pi)$ and composition $1_i\circ _j 1=(-1)^{(j-1)(m_1-1)}$.
\item
For two dioperads $\mc P$ and $\mc Q$, their tensor product $(\mc P\otimes \mc Q)(m,n)=\mc P(m,n)\otimes \mc Q(m,n)$ is again a dioperad with $(\pi,\sigma).(p\otimes q)=(\pi,\sigma). p\otimes (\pi,\sigma).q$ and $(p_1\otimes q_1)\,_i\circ_j(p_2\otimes q_2)=(-1)^{|q_1|\cdot |p_2|}(p_1\,_i\circ_j p_2)\otimes (q_1\,_i\circ_j q_2)$.
\item
Define $\Sigma=\Theta\otimes \Gamma^{-1}$. Then, $\Sigma$ is the endomorphism dioperad on the $1$-dimensional vector space $\kk[1]$ concentrated in degree $-1$, i.e. $\Sigma(m,n)=Hom((\kk[1])^{\otimes n}, (\kk[1])^{\otimes m})$ is concentrated in degree $n-m$.
\item\label{ITM:Omega-dioperad0}
The most relevant ``sign'' dioperad for us is $\Omega:=\Theta^{-1}\otimes \Gamma^{-1}\otimes \Gamma^{-1}$. Note that $\Omega(m,n)$ is a 1-dimensional space concentrated in degree $3-n-2m$ with the $(\Ss_m,\Ss_n)$-action $(\pi,\sigma).1=\sgn(\sigma)$. The composition $_i\circ _j:\Omega(m_1,n_1)\otimes \Omega(m_2,n_2)\to \Omega(m_1+m_2-1,n_1+n_2-1)$ is given by $1_i\circ _j 1=(-1)^{(i-1)(n_2-1)+(m_1-1)(m_2-1)}$.
\end{enumerate}
\end{ex}

\subsection{Quadratic dioperads}

\begin{defn}
Let $T$ be a (finite) directed tree whose sets of vertices is denoted by $\Ve(T)$ and sets of edges is denoted by $\Ed(T)$. For a vertex $v\in \Ve(T)$, denote by $\In(v)$ the set of incoming edges, and by $\Out(v)$ the set of outgoing edges. Vertices $v$ with $(|\Out(v)|,|\In(v)|)=(1,0)$ are called leaves, while vertices $v$ with $(|\Out(v)|,|\In(v)|)=(0,1)$ are called roots; all other vertices are called internal vertices. The set of these types of vertices is denoted by $\Le(T)$, $\Ro(T)$, and $\Int(T)$, respecifvely.

Recall the indexing set $\JJ=\N\times \N_0-\{(1,0),(1,1)\}$ from Definition \ref{DEF:dioperad0}. A directed tree $T$ is a called a tree$_0$, if $(|\Ro(T)|,|\Le(T)|)\in \JJ$, and $T$ has at least one internal vertex, and, furthermore, for each internal vertex $v\in \Int(T)$, we have $(|\Out(v)|,|\In(v)|)\in  \JJ$. Note that if $T$ is a tree$_0$ with $|\Le(T)|>0$ and such that for each internal vertex $v$ of $T$, $|\In(v)|>0$, then $T$ is a reduced tree in the sense of \cite[2.1]{G}.

For $(m,n)\in \JJ$, an $(m,n)$-tree$_0$ is a tree$_0$ $T$ whose roots are labeled by $\{1,\dots,m\}$, and, when $n>0$, whose leaves are labeled by $\{1,\dots,n\}$. Note, that the maximal number of internal vertices in an $(m,n)$-tree$_0$ is $n+2m-3$; in this case, we call an $(m,n)$-tree$_0$ maximally expanded. A maximally expanded tree$_0$ $T$ has all internal vertices $v\in \Int(T)$ of the form $(|\Out(v)|,|\In(v)|)=(1,2)$ or $(|\Out(v)|,|\In(v)|)=(2,0)$. Furthermore, a general $(m,n)$-tree$_0$ $T$ satisfies
\[
\sum_{v\in \Int(T)} \Big(|\In(v)|+2\cdot |\Out(v)|-3\Big) = n+ 2m-3.
\]
\begin{figure}[h]
\[
 \includegraphics[scale=1]{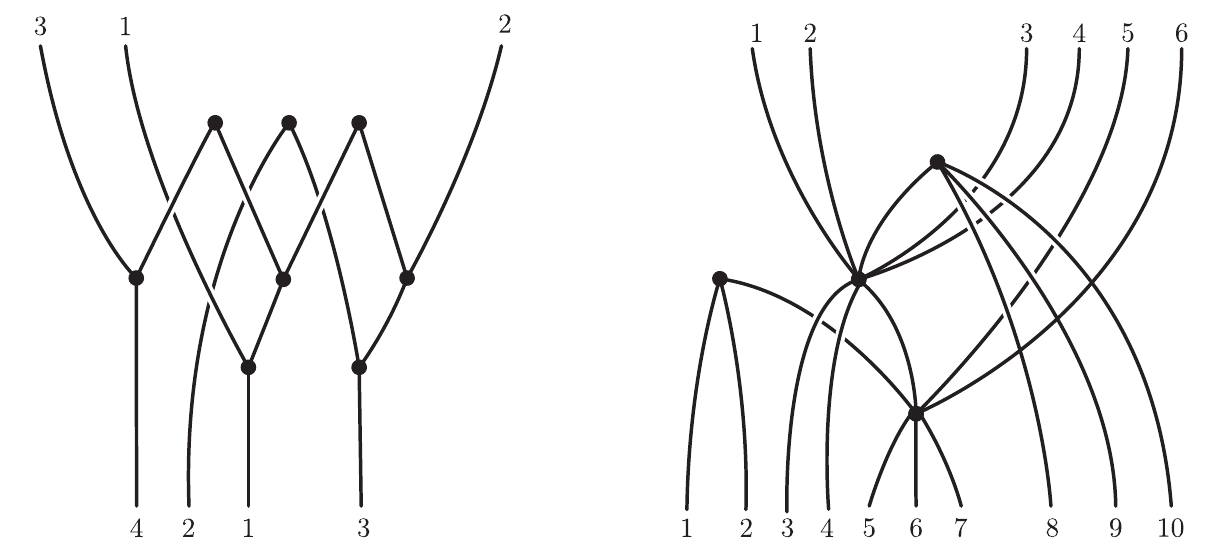} 
\]
\caption{A maximally expanded $(4,3)$-tree$_0$ (left), and a non-maximally expanded $(10,6)$-tree$_0$ (right)}
\end{figure}
\end{defn}

\begin{defn}\label{DEF:quadratic-diop0}
Let $E=\{E(m,n)\}_{(m,n)\in \JJ}$ be a collection of dg-vector spaces, which are $(\Ss_m,\Ss_n)$-bimodules. $E$ can be extended to apply to any two finite sets $X$ and $Y$ in the usual way (see e.g. \cite[Section 1.7]{MSS}) by setting $E(X,Y):= \Bigg( \bigoplus\limits_{ \tiny\begin{array}{c} \text{bijections}\\  f:X\stackrel \sim\to \{1,\dots, m\}\\  g:Y\stackrel \sim\to \{1,\dots, n\}\end{array}} E(m,n)\Bigg)_{(\Ss_m,\Ss_n)}$ for $m=|X|$ and $n=|Y|$. Then, for a tree$_0$ $T$, define, as usual,
\[
E(T)=\bigotimes_{v\in \Int(T)} E(\Out(v),\In(v))
\]
and the free dioperad $\mc F_0(E)$ generated by $E$ to be given by the colimit (cf. \cite[Definition 1.77]{MSS})
\[
\mc F_0(E)(m,n)=\underset{T\in {\bf Iso}((m,n)\text{-tree}_0)}{\text{colim}} E(T),
\]
where ${\bf Iso}((m,n)\text{-tree}_0)$ is the category of $(m,n)\text{-tree}_0$s all of whose morphisms are isomorphisms. Picking a representative $T$ of each isomorphism class of $(m,n)\text{-tree}_0$s, this can also be written as a direct sum over these chosen representatives, $\mc F_0(E)(m,n)=\bigoplus_T E(T)$; cf. \cite[Remark 1.84]{MSS}.

The composition $_i\circ_j$ is given by concatenation of tree$_0$s; the $(\Ss_m,\Ss_n)$-bimodule structure is given by changing the labels of the leaves and roots of the tree$_0$s. The differentials on the $E(m,n)$ induce a differential $d$ on $\mc F_0(E)(m,n)$ so that both composition and symmetric group actions are degree $0$ chain maps. Note that if $E(m,0)=\{0\}$ for all $m\geq 2$, then $\mc F_0(E)=\mc F(E)$ coincides with the free dioperad $\mc F(E)$ as in Gan \cite[2.2]{G}.

Note, that the inclusion $(E\otimes \Omega)(m,n)\hookrightarrow (\mc F_0(E)\otimes \Omega)(m,n)$ induces an isomorphism of dioperads $\mc F_0(E\otimes \Omega)\stackrel \cong\to \mc F_0(E)\otimes \Omega$.

For a dioperad $\mc P$, an ideal $\mc I$ in $\mc P$ is a collection of $(\Ss_m,\Ss_n)$-sub-bimodules $\mc I(m,n)\subset \mc P(m,n)$ such that for $(m_1,n_1),(m_2,n_2)\in \JJ$ with $n_1\neq 0$, $1\leq i \leq n_1$, $1\leq j\leq m_2$, and for any $f\in \mc I(m_1,n_1), g\in \mc P(m_2,n_2)$ or $f\in \mc P(m_1,n_1), g\in \mc I(m_2,n_2)$, we have $f_i\circ_j g\in \mc I(m_1+m+2-1,n_1+n_2-1)$. The quotient $\mc P/\mc I$ of a dioperad $\mc P$ by an ideal $\mc I$ given by $(\mc P/\mc I)(m,n)=\mc P(m,n)/\mc I(m,n)$  is again a dioperad.

Let $E(1,2)$ be a right $\Ss_2$-module and $E(2,0)$ be a left $\Ss_2$-module, which are graded vector spaces possibly in non-zero degree, but with zero differential (cf. co-inner product in Definition \ref{DEF:V-dioperad0}, which is of degree $d$). Assume that all $E(m,n)=\{0\}$ for all $(m,n)\in \JJ-\{(1,2),(2,0)\}$. Note, that in the free dioperad $\mc F_0(E)$ trees with two internal vertices are only of the form $\mc F_0(E)(1,3)$ and $\mc F_0(E)(2,1)$ which have induced left-$\Ss_3$ action, and right-$\Ss_2$-action, respectively.
\[
\begin{array}{rcl}
\mc F_0(E)(1,3) & = & \text{Ind}^{\Ss_3}_{\Ss_2}(E(1,2)\otimes E(1,2)), \\
\mc F_0(E)(2,1) & = & \text{Ind}^{\Ss_2}_{\{1\}}(E(1,2)\otimes E(2,0)).
\end{array}
\]
Let $(R)$ be the ideal in $\mc F_0(E)$ generated by a right $\Ss_3$-submodule $R(1,3)\subset \mc F_0(E)(1,3)$ and a left $\Ss_2$-submodule $R(2,1)\subset \mc F_0(E)(2,1)$. In this case, we denote the quotient $\langle E;R\rangle:=\mc F_0(E)/(R)$ a quadratic dioperad with generators $E$ and relations generated by $R$. Any dioperad of the form $
\mc P=\langle E;R\rangle$ is called a quadratic dioperad.

If $\mc P=\langle E; R\rangle$ is a quadratic dioperad, then the quadratic dual is $\mc P^!=\langle E^\vee;R^\perp\rangle$ where, for $(m,n)=(1,2),(2,0)$, we set $E^\vee(m,n):=E(m,n)^*[-1]\otimes \Omega(m,n)\cong E(m,n)^*\otimes Sgn_n$ and $R^\perp\subset \mc F_0(E^\vee)(2,1)\oplus \mc F_0(E^\vee)(1,3)$ is the orthogonal complement of $R\subset \mc F_0(E)(2,1)\oplus \mc F_0(E)(1,3)$, where we use the identifications
\[
\mc F_0(E^\vee)=\mc F_0(E^*[-1]\otimes \Omega)\cong\mc F_0(E^*[-1])\otimes \Omega\cong \bigoplus_T (\mc F_0(E))^*\otimes  \bigwedge\,^{|\Int(T)|}(\kk[-1])^{\Int(T)} \otimes \Omega.
\]
Note, in particular, that $\mc P^{!!}=\mc P$.
\end{defn}

\subsection{Koszul Duality}

\begin{defn}
Consider a dioperad $\mc P$ so that each $\mc P(m,n)$ is finite dimensional. The $(\Ss_m,\Ss_n)$-bimodules $\mc P(m,n)$ dualize to $(\Ss_m,\Ss_n)$-bimodules $\mc P(m,n)^*=Hom(\mc P(m,n),\kk)$ with the dual representation, i.e. $((\pi,\sigma).f)(x)=f((\pi^{-1},\sigma^{-1}).x)$ for $f\in \mc P(m,n)^*, x\in \mc P(m,n)$. Shifting these modules up by one gives $(\Ss_m,\Ss_n)$-bimodules $\mc P(m,n)^*[-1]$, so that we may take the free dioperad $\mc F_0(\mc P^*[-1])$ on the collection $\mc P^*[-1]=\{\mc P(m,n)^*[-1]\}_{(m,n)\in \JJ}$. Note, that we may factor the shifts ``$[-1]$'' via the identity $(\mc P^*[-1])(T)\cong \mc P^*(T)\otimes \bigwedge^{|\Int(T)|}(\kk[-1])^{\Int(T)}$, where the second factor is the top exterior power on the space of internal vertices of $T$ each of which being in degree $1$.

Next, we define cobar complex $\bC \mc P$ to be $\mc F_0(\mc P^*[-1])$ with the same composition and $(\Ss_m,\Ss_n)$-bimodule structure as $\mc F_0(\mc P^*[-1])$, but whose differential has one more component $\delta=d+d^\circ$. To describe $d^\circ$, dualize the compositions $_i\circ _j:\mc P(m_1,n_1)\otimes \mc P(m_2,n_2)\to \mc P(m_1+m_2-1, n_1+ n_2-1)$ to obtain cocompositions $_i\Delta _j:\mc P(m_1+m_2-1, n_1+n_2-1)^*\to \mc P(m_1,n_1)^*\otimes \mc P(m_2,n_2)^*$ which are chain maps (with respect to $d^*$) of degree $0$. Furthermore, there is a degree $1$ map on the $\bigwedge^{|\Int(T)|}(\kk[-1])^{\Int(T)}$-components of  $(\mc P^*[-1])(T)$ given by replacing a vertex in a tree$_0$ by $v\wedge w$, where $v$ and $w$ are the new vertices obtained from the cocomposition, such that, by convention, $v$ has an edge outgoing into $w$. Putting these maps together induces a differential $d^\circ:\mc F_0(\mc P^*[-1])\to \mc F_0(\mc P^*[-1])$ of degree $1$ so that $(d^\circ)^2=0$ and $dd^\circ+d^\circ d=0$, and thus $\delta^2=0$. Note further, that $d^\circ$ commutes with the composition in $\mc F_0(\mc P^*[-1])$ (concatenation of trees), as well as the $(\Ss_m,\Ss_n)$-action in $\mc F_0(\mc P^*[-1])$ (relabeling of the leaves and roots).

With this, we define the cobar dual dioperad $\bD \mc P$ as the tensor product of $\bC \mc P$ with  $\Omega$ from Example \ref{EX:sign-dioperad0s}\eqref{ITM:Omega-dioperad0}
\[
\bD \mc P :=\bC \mc P \otimes \Omega,
\]
As a dg-vector space, $\bD \mc P(m,n)=\bC \mc P(m,n) \otimes \Omega(m,n)\cong \mc F_0(\mc P^*[-1])(m,n)\otimes \kk[n+2m-3]$\label{PAGE:DP(m,n)} is a direct sum over all $(m,n)$-tree$_0$s $T$ of the spaces $\mc P^*(T)\otimes \bigwedge^{|\Int(T)|}(\kk[-1])^{\Int(T)}\otimes \kk[n+2m-3]$. We will denote the latter factors as $\Det(T):=\bigwedge^{|\Int(T)|}(\kk[-1])^{\Int(T)}\otimes \kk[n+2m-3]$, which is a one-dimensional space concentrated in degree $3+|\Int(T)|-n-2m$, where the number $|\Int(T)|$ of internal vertices in an $(m,n)$-tree$_0$ can range from $1$ to $n+2m-3$. While the differential $d$ only increases the degree in $\mc P^*(T)$, the differential $d^\circ$ increases the number of internal vertices:
\begin{equation}\label{EQN:DP-complex}
\bigoplus_{|\Int(T)|=1} \mc P^*(T)\otimes \Det(T)\stackrel{d^\circ}\to \bigoplus_{|\Int(T)|=2} \mc P^*(T)\otimes \Det(T)\stackrel{d^\circ}\to \dots \stackrel{d^\circ}\to \bigoplus_{|\Int(T)|=n+2m-3} \mc P^*(T)\otimes \Det(T)
\end{equation}
Thus, $\bD\mc P=(\bD\mc P)^{\bullet,\bullet}$ is bigraded with internal degree in $\mc P^*$ (raised by $d$) and vertex degree (raised by $d^\circ$) ranging from $4-n-2m$ to $0$ in the above complex.
\end{defn}
We have the following standard facts.
\begin{prop}
If $\mc P\to \mc Q$ is a quasi-isomorphism of (finite dimensional) dioperads, then there is an induced quasi-isomorphism $\bD\mc Q\to \bD\mc P$.
\end{prop}
\begin{proof}
This follows just as in \cite[Theorem (3.2.7)(b)]{GK} or \cite[Corollary 3.13]{MSS}. If $\mc P(m,n)$ and $\mc Q(m,n)$ are finite dimensional, the maps maps $\mc Q^*(T)\to \mc P^*(T)$ are quasi-isomorphic, and thus so is the induced map $\bD\mc Q\to \bD\mc P$.
\end{proof}
\begin{prop}
There is a quasi-isomorphism $\bD\bD \mc P\to \mc P$.
\end{prop}
\begin{proof}
The proof is standard, and just as in \cite[Proposition 3.3]{G}, going back to \cite[Theorem (3.2.16)]{GK} and described in more detail in \cite[Theorem 3.24]{MSS}. It uses on the fact, that $\bD \bD\mc P$ is given by a direct sum over tree$_0$s $T_v$ placed inside a tree$_0$ $T$ at each internal vertex $v$ of $T$; denote the grafted tree$_0$ by $S=T(T_{v_1},\dots, T_{v_{|\Int(T)|}})$,
\[
\bD\bD\mc P(m,n)= \bigoplus_{(m,n)-\text{tree$_0$s } T}\bigoplus_{\tiny\begin{array}{c}\forall v_i\in \Int(T):\\ (\Out(v_i),\In(v_i))-\text{tree$_0$s } T_{v_i}\end{array}} \mc P(S)\otimes \Det(T)\otimes \bigotimes_{v_i\in \Int(T)} \Det(T_{v_i}).
\]
Then, $(\bD \bD\mc P)^{\bullet,\bullet, \bullet}$ has a triple grading, where the first differential $d$ raises the degree in $\mc P$, the second differential $d_{\bD \mc P}$ collapses an edge in a tree$_0$ $T_v$ (according to the dioperad rules of $\mc P$), and the third differential $d^\circ$ combines two tree$_0$s $T_{v_1}$ and $T_{v_2}$ of two vertices $v_1,v_2$ of $T$ by attaching $T_{v_1}$ to $T_{v_2}$ and thus leaving $S$ invariant. If $S$ has more than one internal vertex, then $\bD \bD\mc P$ with differential $d^\circ$ only changes the way we partition $S$ into subtree$_0$s $T_{v_i}$, and is thus isomorphic to the augmented cellular chains of a simplex, which is acyclic. If $S$ has only one internal vertex, then $d^\circ$ is zero, while $d_{\bD \mc P}$ also vanishes for this case. The argument is then as in \cite{GK} and \cite{MSS}.
\end{proof}

\begin{cor}\label{COR:unique-resolution}
If $\mc P, \mc Q, \mc R$ are dioperads with zero differentials, and there are quasi-isomorphisms $\bD Q\to \mc P$ and $\bD \mc R\to \mc P$, then $\mc Q\cong \mc R$ are isomorphic.
\end{cor}
\begin{proof}
Applying ``$\bD$'' to $\bD Q\to \mc P\leftarrow \bD \mc R$ induces quasi-isomorphisms $\mc Q\leftarrow\bD\bD Q\to\bD \mc P\leftarrow \bD\bD \mc R\to \mc R$, thus an isomorphism on homology $\mc Q=H^\bullet(\mc Q)\cong H^\bullet(\mc R)=\mc R$.
\end{proof}

\begin{defn}
Let $\mc P$ be a quadratic dioperad. Taking homology in the right-most term of \eqref{EQN:DP-complex} (i.e. in vertex degree $0$), yields precisely the quadratic dual dioperad $\mc P^!\cong H^0(\bD \mc P)$. Thus, there is a map of dioperads $\bD \mc P\to \mc P^!$ by mapping \eqref{EQN:DP-complex} to $\mc P^!$. Now, $\mc P$ is called Koszul, if this map $\bD \mc P\to \mc P^!$ is a quasi-isomorphism, i.e. $H^\bullet(\bD \mc P)$ has all of its homology in $H^0(\bD \mc P)\cong \mc P^!$. By the previous Corollary \ref{COR:unique-resolution}, $\mc P^!$ gives the unique dioperad (unique up to isomorphism) whose cobar dual is quasi-isomorphic to $\mc P$.
\end{defn}

\section{Associative algebras with co-inner products}\label{SEC:Vinfty}

We now examine the dioperad $\mc V^{(d)}$. We show that $(\mc V^{(d)})^!=H^0(\bD \mc V)$ is $(\mc V^{(d)})^!=\mc V^{(-d)}$ and that $\mc V^{(d)}$ is Koszul, i.e. $\bD (\mc V^{(d)})^!\stackrel \sim\to \mc V^{(d)}$ is a quasi-isomorphism. Setting $\mc V^{(d)}_\infty:=\bD \mc V^{(d)}$, we also remark that the notion of $\mc V^{(d)}_\infty$ algebras has already appeared in \cite{TZ}.

\subsection{The $\mc V^{(d)}$ dioperad}

\begin{defn}\label{DEF:V-dioperad0}
Let $d\in \Z$. We define $\mc V^{(d)}$ as the $d$-dimensional quadratic dioperad $\mc V^{(d)}=\langle E_{\mc V^{(d)}},R_{\mc V^{(d)}}\rangle$ generated by $E_{\mc V^{(d)}}(1,2)=\kk\cdot \mu\oplus \kk\cdot \bar \mu$ with right $\Ss_2$-action $\sigma.\mu=\bar \mu$ interchanging $\mu$ and $\bar \mu$, and $E_{\mc V^{(d)}}(2,0)=\kk[-d]\cdot \nu$ be concentrated in degree $d$ with left $\Ss_2$-action $\sigma.\nu=\nu$. The space of relations is spanned by associativity of $\mu$ and invariance of $\nu$,
\[
\mu _1\circ_1\mu=\mu _2\circ_1\mu,\quad\quad \mu _1\circ_2\nu=\mu _2\circ_1 \nu.
\]
\begin{figure}[h]
\[
 \includegraphics[scale=.8]{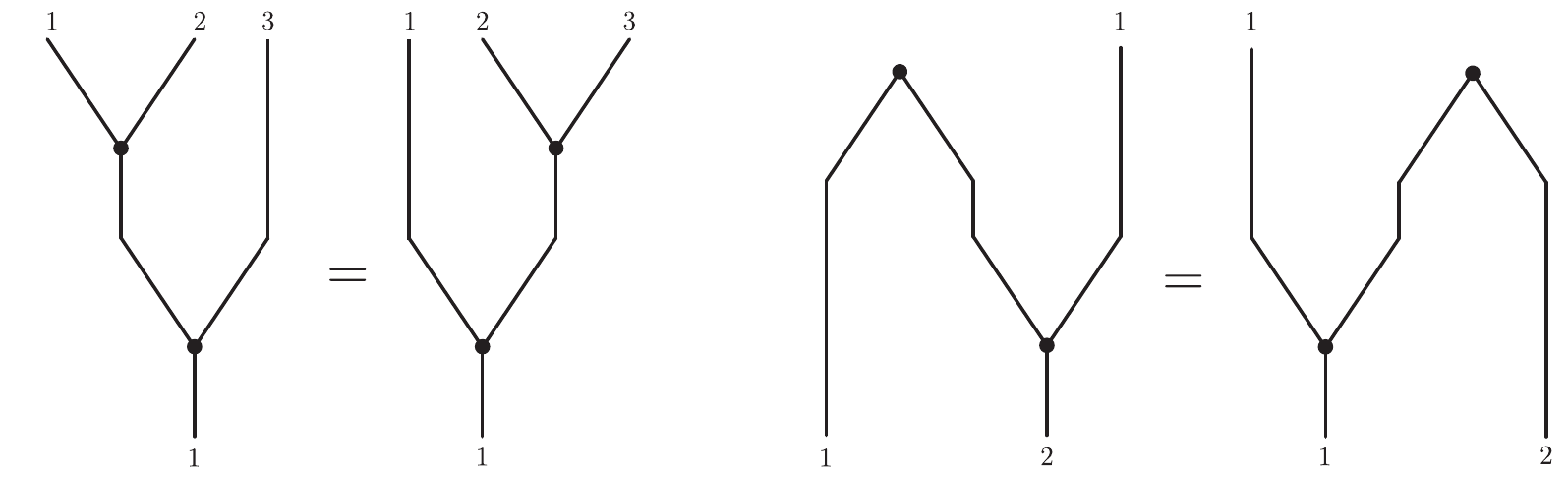} 
\]
\caption{The relations $\mu _1\circ_1\mu=\mu _2\circ_1\mu$ and $\mu _1\circ_2\nu=\mu _2\circ_1 \nu$}\label{FIG:V-relations}
\end{figure}

Recall from Definition \ref{DEF:quadratic-diop0}, that the quadratic dioperad is given by the free dioperad modulo its relations, i.e. $\mc V^{(d)}(m,n)=\mathcal F_0(E_{\mc V^{(d)}})(m,n)/(R_{\mc V^{(d)}})(m,n)$. Here, $\mathcal F_0(E_{\mc V^{(d)}})$ is given by a sum over tree$_0$s with internal vertices only of type $(1,2)$ and $(2,0)$. It will be useful for our purposes to represent the tree$_0$s $T$ of $\mc V^{(d)}$ by planar trees (with a flow from top to bottom) by picking the choice of $\mu$ (instead of $\bar\mu$) in the decoration of the vertices $v\in\Int(T)$ by $E_{\mc V^{(d)}}(\Out(v),\In(v))=\Big( \bigoplus\limits_{ \tiny g:\In(v)\stackrel \sim\to \{1,2\}} E_{\mc V^{(d)}}(1,2)\Big)_{\Ss_2}$ via the convention that the left (respectively right) edge incoming to $v$ is $g^{-1}(1)$ (respectively $g^{-1}(2)$); see e.g. Figure \ref{FIG:V-relations}. 
(Although there appears to be an ambiguity in representing the planar tree concerning $\nu\in E_{\mc V^{(d)}}(2,0)$, we will see in Proposition \ref{PROP:V(m,n)-as-gammas}\eqref{ITM:gamma-cyclic} that this will not matter for our purposes.)
\end{defn}

We first calculate the quadratic dual of $\mc V^{(d)}$.
\begin{prop}
The quadratic dual of $\mc V^{(d)}$ is $(\mc V^{(d)})^!=\mc V^{(-d)}$.
\end{prop}
\begin{proof}
The quadratic dual is generated by $E_{(\mc V^{(d)})^!}(1,2)=\kk\cdot \mu^*\oplus \kk\cdot \bar \mu^*$ with right $\Ss_2$-action $\sigma.\mu^*=-\bar \mu^*$ interchanging $\mu^*$ and $-\bar \mu^*$, and $E_{(\mc V^{(d)})^!}(2,0)=\kk[-d]\cdot \nu^*$ be concentrated in degree $-d$ with trivial left $\Ss_2$-action $\sigma.\nu^*=\nu^*$ (recall that the $(\Ss_m,\Ss_n)$-action on $\Omega(m,n)$ is trivial on $\Ss_m$ and the sign representation on $\Ss_n$). The operation $\mu^*$ satisfies associativity, just as in the usual calculation for the associative operad. It remains to check that we also have the relation $\mu^* \,_1\circ_2\nu^*=\mu^* \,_2\circ_1 \nu^*$ in $(\mc V^{(d)})^!$.

Using that $\mc F_0(E_{\mc V^{(d)}}^\vee)(2,1)\cong \bigoplus_{(2,1)\text{-tree}_0\text{s } T} E_{\mc V^{(d)}}^*(T)\otimes  \bigwedge\,^{|\Int(T)|}(\kk[-1])^{\Int(T)} \otimes \Omega(2,1)$, we can calculate its differential as
\begin{equation}\label{EQN:boundary-(2,1)}
 \includegraphics[scale=.8]{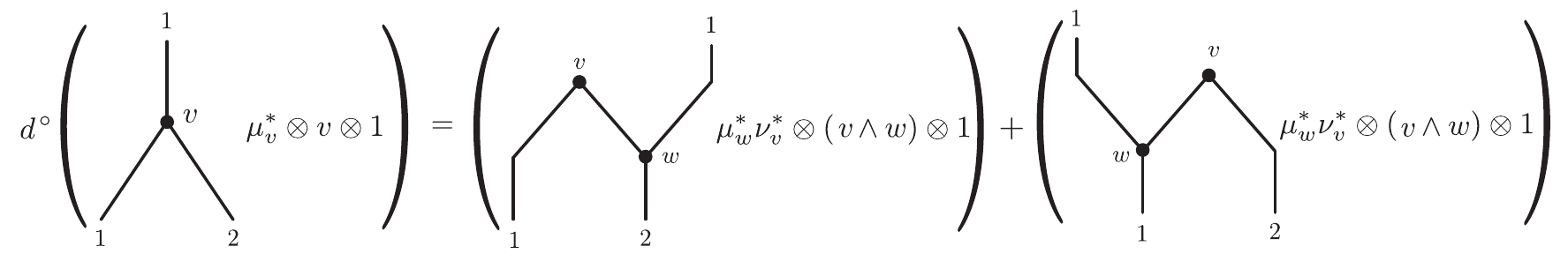} 
\end{equation}
On the other hand, composing operations $\mc F_0(E_{\mc V^{(d)}}^\vee)(1,2)$ and $\mc F_0(E_{\mc V^{(d)}}^\vee)(2,0)$ yields.
\begin{equation}\label{EQN:composition-(2,1)}
 \includegraphics[scale=.8]{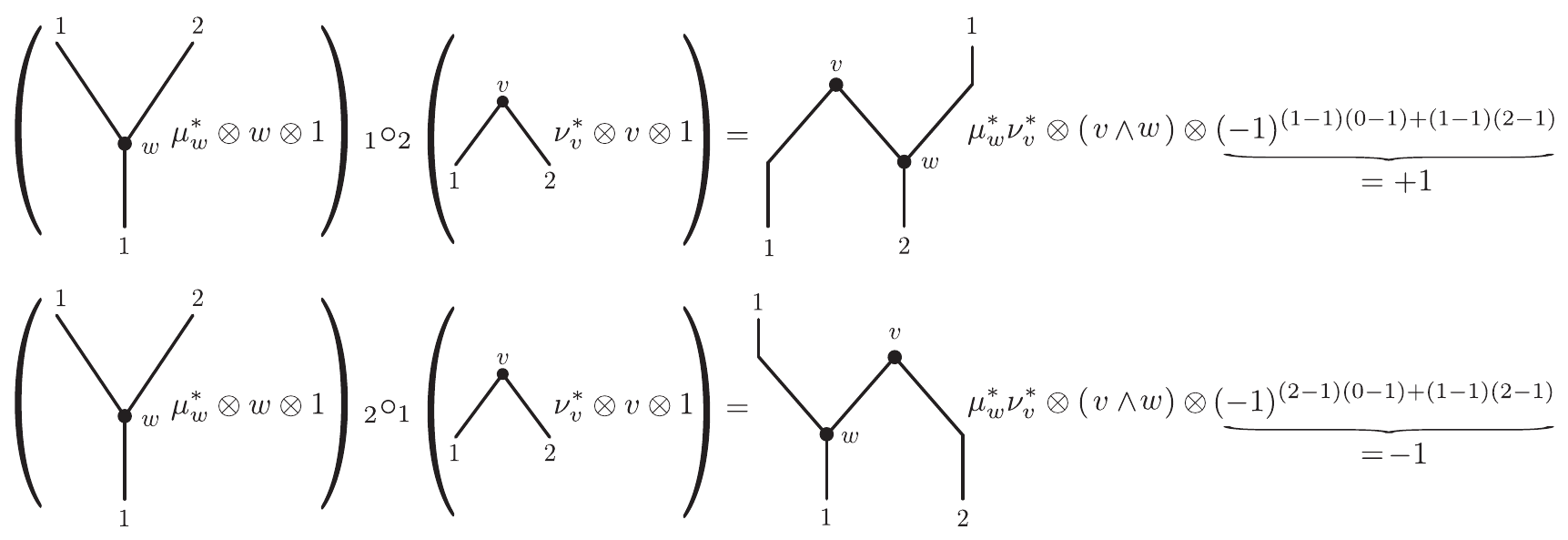} 
\end{equation}
Comparing \eqref{EQN:boundary-(2,1)} with \eqref{EQN:composition-(2,1)}, we see that $\mu^* \,_1\circ_2\nu^*-\mu^* \,_2\circ_1 \nu^*=0$ holds in $H^0(\bD \mc V^{(-d)})$.
\end{proof}

Next, we will explicitly identify the spaces $\mc V^{(d)}(m,n)$. To this end we first define certain ``generating'' operations $\gamma_\alpha$.

\begin{defn}\label{DEF:A(m,n)-gamma}
Let $\alpha:\Z_{n+m}\to \{\oo,\ii\}$ be a map from the cyclic group $\Z_{n+m}$ which maps precisely $m$ elements to the outgoing label ``$\oo$'' and precisely $n$ elements to the incoming label ``$\ii$''. We sometimes write $\alpha$ as a sequence $\alpha=(\alpha_0\dots \alpha_{n+m-1})$, where $\alpha_j:=\alpha(j)$. Now, $\alpha$ is called $(m,n)$-labeled if there is an injection $\ooo:\{1,\dots, m\}\to\Z_{n+m}$ so that the $\alpha_{\ooo(j)}$ are precisely the outgoing labels $\alpha_{\ooo(j)}=\oo$, and there is an injection $\iii:\{1,\dots, n\}\to\Z_{n+m}$ so that the $\alpha_{\iii(j)}$ are precisely the incoming labels $\alpha_{\iii(j)}=\ii$. (Note that $\alpha$ is actually determined by $\ooo$ or $\iii$.) We will need to consider such $(m,n)$-labeled $\alpha$ up to cyclic rotation. More precisely, let $\tau:\Z_{n+m}\to \Z_{n+m}, \tau(p)\equiv p+1$ (mod $n+m$) be the cyclic rotation. We can use $\tau$ to define an equivalence relation on the set of triples $(\alpha, \ooo, \iii)$, by letting it be generated by $(\alpha, \ooo,\iii)\sim(\alpha\circ \tau, \tau^{-1}\circ\ooo, \tau^{-1}\circ \iii)$. The equivalence classes of the induced equivalence relation are denoted by $\llbracket\alpha,\ooo,\iii\rrbracket$. The set of all such equivalence classes is denoted by $\mathfrak A(m,n)=\{\llbracket\alpha,\ooo,\iii\rrbracket: \ooo,\iii\text{ is an $(m,n)$-labeling of } \alpha\}$. 

Let $\alpha:\Z_{n+m}\to \{\oo,\ii\}$  be a sequence of labels as above with $\alpha_{0}=\oo$. Assume that the outgoing labels are precisely $\alpha_{p_1}=\alpha_{p_2}=\dots=\alpha_{p_m}=\oo$ with $0= p_1<p_2<\dots<p_m< n+m$. Denote by $\gamma_\alpha$ the operation $\gamma_\alpha:=(\dots(((\mu_1\circ_1\mu_1 \circ_1\dots _1\circ_1\mu)_{p_{2}} \circ_1\nu)_{p_{3}-1}\circ_1\nu)\dots )_{p_{m}-(m-2)}\circ_1 \nu\in \mc V^{(d)}(m,n)$ depicted in Figure \ref{FIG:gamma-alpha}.
\begin{figure}[h]
\[
 \includegraphics[scale=1]{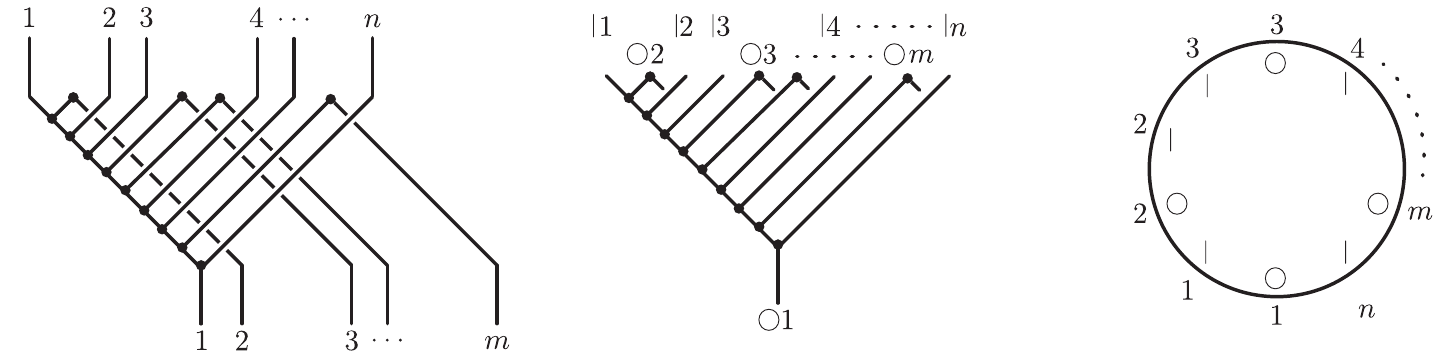} 
\]
\caption{We represent the operation $\gamma_\alpha\in \mc V^{(d)}(m,n)$ for $\alpha=(\oo\ii\oo\ii\ii\oo\oo\ii\ii\oo\ii)$, defined as a composition of $\mu$s and $\nu$s as in the left figure or, more compactly, as in the middle figure. We will see in Proposition \ref{PROP:V(m,n)-as-gammas}(\ref{ITM:V(m,n)=gammas}) that $\gamma_\alpha$ is precisely determined by the class $\llbracket\alpha,\ooo,\iii\rrbracket$, which we will represent as in the right figure.
}\label{FIG:gamma-alpha}
\end{figure}
Note, that $\gamma_\alpha$ has an implied labeling of its roots and leaves, which induces an $(m,n)$-labeling $\ooo_{\alpha}$, $\iii_{\alpha}$ as follows. If the outgoing labels are at the positions $0= p_1<p_2<\dots<p_m< n+m$, then define $\ooo_{\alpha}:\{1,\dots, m\}\to\Z_{n+m}$ by $\ooo_{\alpha}(j)=p_j$. If the incoming labels are at the positions $0< q_1<q_2<\dots<q_n< n+m$, then define $\iii_{\alpha}:\{1,\dots, n\}\to\Z_{n+m}$ by $\iii_{\alpha}(j)=q_j$.

Moreover, applying $\pi\in\Ss_m$ and $\sigma\in\Ss_n$ to $\gamma_\alpha$ from above yields the operation $(\pi,\sigma).\gamma_\alpha$, which is given by the same tree$_0$ as the one for $\gamma_\alpha$, but with a different labeling of its roots and leaves. This labeling can be encoded via the new $(m,n)$-labeling given by $\ooo_{\alpha}\circ\pi^{-1}$ and $\iii_{\alpha}\circ\sigma$. (Indeed, any $\ooo$ and $\iii$ is given by some $\alpha$, $\pi$ and $\sigma$ in this way.)
\end{defn}

We now show that $\mc V^{(d)}$ is essentially given by the operations $(\pi,\sigma).\gamma_\alpha$  from above.
\begin{prop}\label{PROP:V(m,n)-as-gammas} \quad
\begin{enumerate}
\item\label{ITM:comp-gamma}
Every composition of $\mu$s and $\nu$s in $\mc V^{(d)}(m,n)$ is equal to $(\pi,\sigma).\gamma_\alpha$ for some $\alpha:\Z_{n+m}\to \{\oo,\ii\}$ and some $\pi \in \Ss_m$, $\sigma \in \Ss_n$.
\item\label{ITM:gamma-cyclic}
Two $(\pi,\sigma).\gamma_\alpha$ and $(\pi',\sigma').\gamma_{\alpha'}$ are equal in $\mc V^{(d)}(m,n)$ if and only if $(\alpha,\ooo_\alpha\circ \pi^{-1}, \iii_\alpha\circ \sigma)\sim(\alpha',\ooo_{\alpha'}\circ {\pi'}^{-1}, \iii_{\alpha'}\circ \sigma')$. In this case, we denote this element by $\gamma_{\llbracket\alpha,\ooo_\alpha\circ \pi^{-1}, \iii_\alpha\circ \sigma\rrbracket}=(\pi,\sigma).\gamma_\alpha=(\pi',\sigma').\gamma_{\alpha'}$.
\item\label{ITM:V(m,n)=gammas}
The space $\mc V^{(d)}(m,n)$ is isomorphic to $\bigoplus_{\llbracket\alpha,\ooo,\iii\rrbracket\in \mathfrak A(m,n)} \kk[d(1-m)]\cdot \gamma_{\llbracket\alpha,\ooo,\iii\rrbracket}$, which has dimension $|\mathfrak A(m,n)|$ and is concentrated in degree $d(m-1)$.
\end{enumerate}
\end{prop}
\begin{proof}
For item \eqref{ITM:comp-gamma}, let $\rho\in \mc V^{(d)}(m,n)$ be given by generators $\mu$s, $\bar\mu$s, and $\nu$s decorating the vertices of an tree$_0$ $T$ together with a labeling of its roots and leaves. By using the $\Ss_2$-action, we may assume that there are no $\bar\mu$s in $T$, and that $T$ is a planar tree (cf. Definition \ref{DEF:V-dioperad0}). Fix a root $x$ in $T$. Using the relations in $\mc V^{(d)}$, we show that the tree $T$ is equivalent to a tree of the form 
\begin{equation}\label{EQN:tree-iteration}
 \includegraphics[scale=.8]{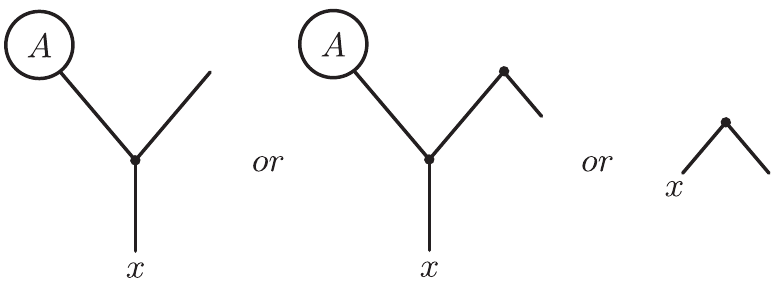} 
\end{equation}
where the $A$ in the circle denotes itself a planar tree which is decorated by $\mu$s and $\nu$s. Item \eqref{ITM:comp-gamma} then follows from this by inductively repeating this procedure for $A$ itself until $\rho$ is of the form $\gamma_\alpha$ together with a labeling of its roots and leaves, which is induced by applying some $\pi \in \Ss_m$ and $\sigma \in \Ss_n$.

To check $\rho$ is of the form \eqref{EQN:tree-iteration}, note that the root $x$ is either the output of a product $\mu$ or a co-inner product $\nu$:
\[
 \includegraphics[scale=.8]{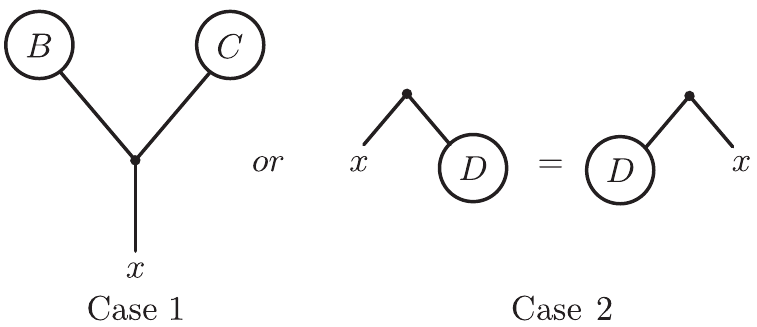} 
\]
where, again $B$, $C$, $D$ denote decorated sub-trees. In the first case, we can successively move more and more from $C$ over to $B$ via the relations
\[
 \includegraphics[scale=.8]{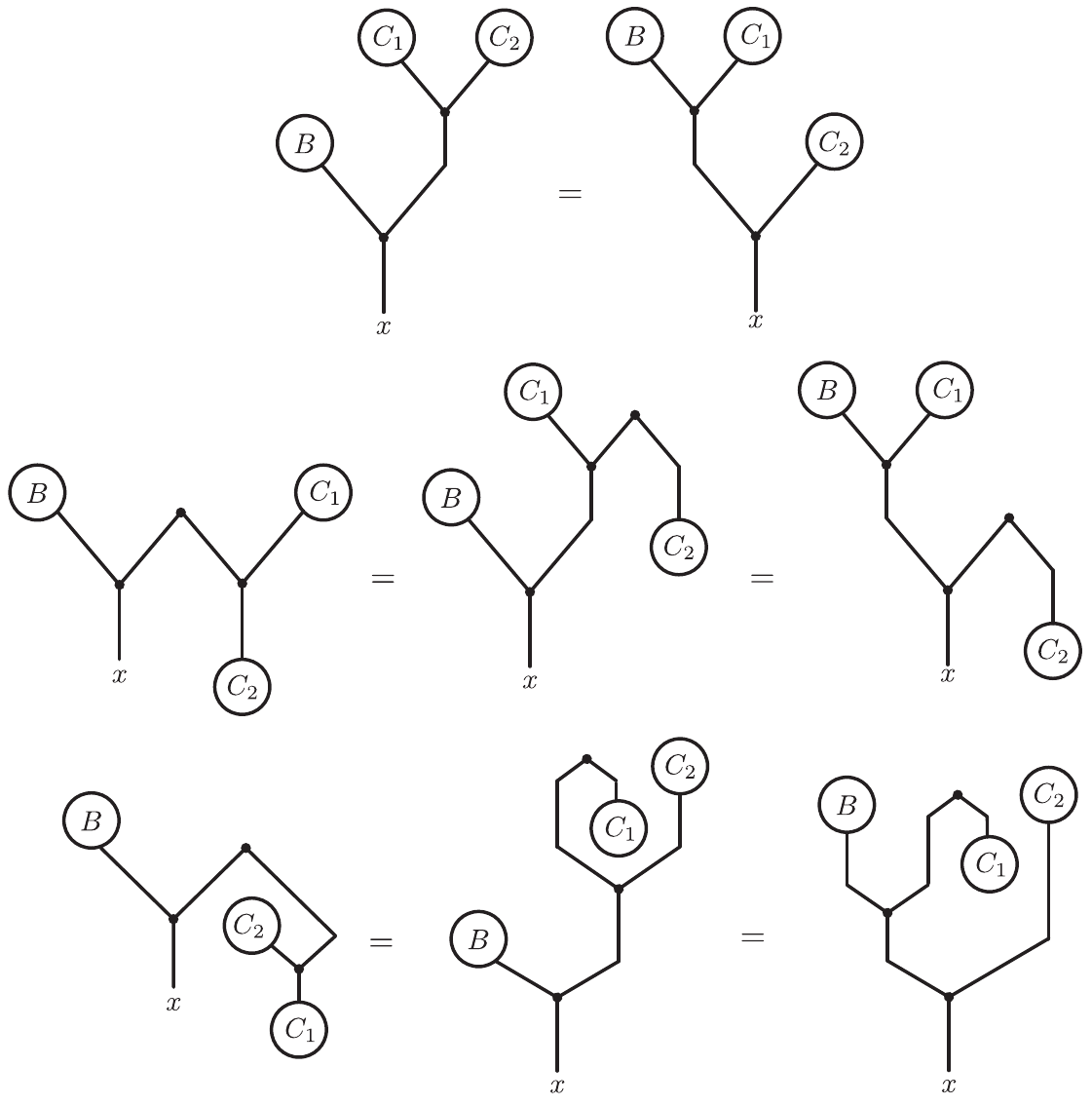} 
\]
The second case, when $D$ is not trivial, may be reduced to the first via the relations
\[
 \includegraphics[scale=.8]{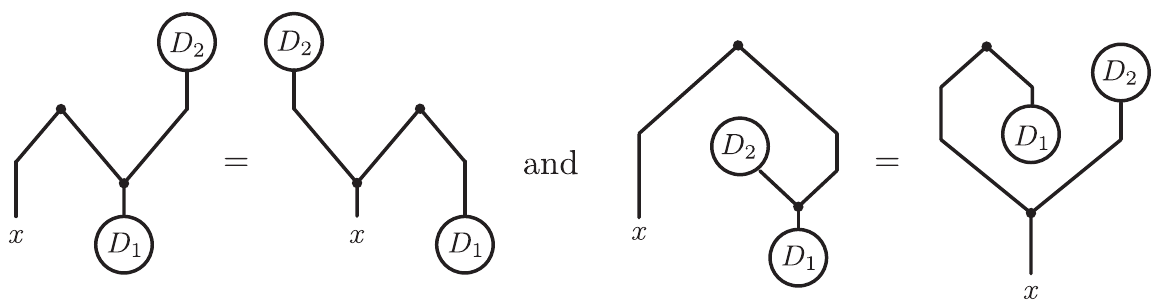} 
\]

For item \eqref{ITM:gamma-cyclic}, we first note that to any $\rho\in \mc V^{(d)}(m,n)$ given by generators $\mu, \bar \mu, \nu$ decorating an $(m,n)$-tree$_0$, there is a well defined equivalence class $\llbracket\alpha,\ooo,\iii\rrbracket$ associated to $\rho$. To see this, write $\rho$ as a planar tree $T$ decorated only with $\mu$s and $\nu$s as in Definition \ref{DEF:V-dioperad0}. Starting from any root or leaf of $T$, proceeding in a clockwise direction along the tree, we can read off outputs $\alpha(j):=\oo$ or inputs $\alpha(j):=\ii$ together with their labels to obtain such data $\alpha$, $\ooo$, and $\iii$. For example, for $\gamma_\alpha$ starting from the root $1$, we obtain exactly $\alpha$, $\ooo_\alpha$, and $\iii_\alpha$ (cf. Figure \ref{FIG:gamma-alpha}), while for $(\pi,\sigma).\gamma_\alpha$ we obtain $\alpha$, $\ooo_\alpha\circ \pi^{-1}$, and $\iii_\alpha\circ \sigma$. To see that this data is well-defined as an equivalence class $\llbracket\alpha,\ooo,\iii\rrbracket$, note that any possible ambiguity in writing $\rho$ does not change the equivalence class $\llbracket\alpha,\ooo,\iii\rrbracket$; these ambiguities are: choosing a starting root or leaf of $T$, presenting the planar tree $T$ by switching the order of $\nu$ via the $\Ss_2$-action, as well as applying the relations defining $\mc V^{(d)}$ in Figure \ref{FIG:V-relations}. 
This shows that if $(\pi,\sigma).\gamma_\alpha=(\pi',\sigma').\gamma_{\alpha'}$, then their associated classes must be equal, $\llbracket\alpha,\ooo_\alpha\circ \pi^{-1}, \iii_\alpha\circ \sigma \rrbracket=\llbracket\alpha',\ooo_{\alpha'}\circ {\pi'}^{-1}, \iii_{\alpha'}\circ \sigma'\rrbracket$ (the ``only if'' part of the statement).

To see the converse, assume $(\alpha,\ooo_\alpha\circ \pi^{-1}, \iii_\alpha\circ \sigma)\sim(\alpha',\ooo_{\alpha'}\circ {\pi'}^{-1}, \iii_{\alpha'}\circ \sigma')$, i.e. $(\alpha,\ooo_\alpha\circ \pi^{-1}, \iii_\alpha\circ \sigma)=(\alpha'\circ \tau^k,\tau^{-k}\circ \ooo_{\alpha'}\circ {\pi'}^{-1}, \tau^{-k}\circ \iii_{\alpha'}\circ \sigma')$ for some integer $k$. Thus, in particular $\alpha=\alpha'\circ \tau^{k}$, so that there is an output label in $\alpha$ at the $(n+m-k)$th spot, since  $\alpha(n+m-k)=\alpha'(\tau^k(n+m-k))=\alpha'(0)=\oo$. To keep track of these outputs, we will label the root of $\gamma_\alpha$ at the $0$th spot with ``$x$'' and the root at the $(n+m-k)$th spot with the letter ``$y$'' in the figures below. In the case where $k=1$, there is an output label at $\alpha_{n+m-1}=\oo$, and we may use the identities
\begin{equation}\label{EQN:gamma-cyclic-last-spot}
 \includegraphics[scale=0.8]{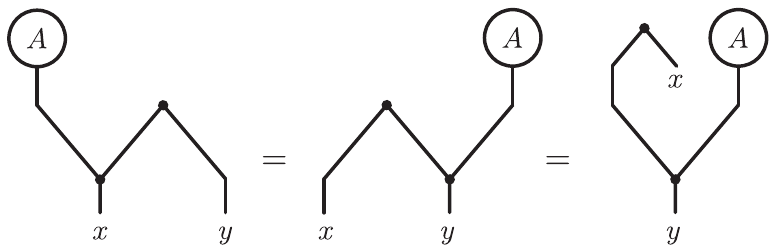} 
\end{equation}
where $A$ itself is an operation as in Figure \ref{FIG:gamma-alpha}. Note, that the right-hand side of \eqref{EQN:gamma-cyclic-last-spot} may again be written in the form of Figure \ref{FIG:gamma-alpha} where the $0$th root is now ``$y$'' instead of ``$x$'' yielding the graph $\gamma_{\alpha'}$ with a cyclic reordering of its labels. In the case $k=n+m-1$, the same equation \eqref{EQN:gamma-cyclic-last-spot} may be used but in a reversed order (switching $x$ and $y$). In the case where $1<k<n+m-1$, we may take the identities
\begin{equation}\label{EQN:gamma-cyclic-mid-spot}
 \includegraphics[width=10cm]{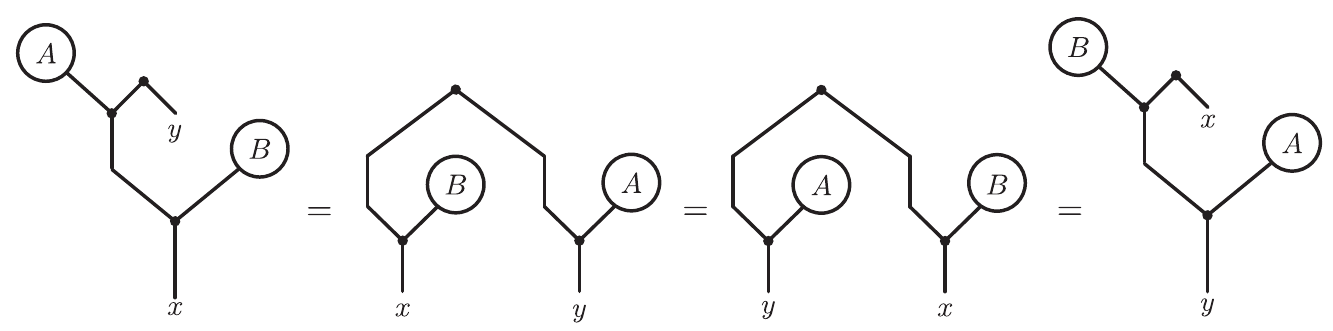} 
\end{equation}
where both $A$ and $B$ are as in Figure \ref{FIG:gamma-alpha}, and the right-hand side can again be rewritten as in Figure \ref{FIG:gamma-alpha} with the $0$th root ``$y$'' instead of ``$x$.'' In the above identities \eqref{EQN:gamma-cyclic-last-spot} and \eqref{EQN:gamma-cyclic-mid-spot}, the associated label cyclically rotate as appropriate, showing that $(\pi,\sigma).\gamma_\alpha=(\pi',\sigma').\gamma_{\alpha'}$.

Item \eqref{ITM:V(m,n)=gammas} now follows from \eqref{ITM:comp-gamma} and \eqref{ITM:gamma-cyclic}.
\end{proof}

\subsection{The $\mc V_\infty^{(d)}$ dioperad}

Using the explicit description of $\mc V^{(d)}$ in Proposition \ref{PROP:V(m,n)-as-gammas} we can calculate the homology $H^\bullet(\bD \mc V^{(d)})$ and see that it has no ``higher'' homology.
\begin{thm}\label{THM:V-is-Koszul}
There is a quasi-isomorphism of dioperads $\bD\mc V^{(d)}\to \mc V^{(-d)}$.
\end{thm}
The remainder of this subsection is concerned with the proof of Theorem \ref{THM:V-is-Koszul}. We need to show that the homology of $H^\bullet(\bD \mc V^{(d)})$ is concentrated in the top degree, which we will do by isomorphically identifying $\bD \mc V^{(d)}(m,n)$ (up to a shift) with the cellular chains $C_\bullet(Z(m,n))$ on a disjoint union union of contractible spaces, $Z(m,n)=\coprod\limits_{\llbracket\alpha,\ooo,\iii\rrbracket\in \mathfrak A(m,n)} Z_\alpha$. Here, $Z_\alpha$ is the assocoipahedron defined in \cite{PT}, which is a contractible cell complex of dimension $dim(Z_\alpha)=n+2m-4$. More precisely, it is a cellular subdivision of an associahedron times a simplex, which has the property, that its cells are indexed by directed planar trees with vertices whose number of incoming and outgoing edges are as in $\mc V^{(d)}$, i.e. $(|\Out(v)|,|\In(v)|)\in \JJ$. (Note, that the cell complex $Z_\alpha$ associated to a classes $\llbracket\alpha,\ooo,\iii\rrbracket$ depends only on $\alpha$ but not on the labelings $\ooo$, $\iii$.) For a given $\alpha$, we denote by $t^{(j)}_{\alpha,1},\dots,t^{(j)}_{\alpha,r}$ the directed planar trees indexing the cells of $Z_\alpha$ of degree $j\in\{0,\dots, n+2m-4\}$, where $r=r(\alpha,j)$ is the number of $j$-cells in $Z_\alpha$. Thus, the $t^{(j)}_{\alpha,i}$ provide a basis for the cellular chains on $Z_\alpha$, and thus for $Z(m,n)$, as a graded vector space:
\begin{equation}
C_\bullet(Z(m,n))=\bigoplus_{\llbracket\alpha,\ooo,\iii\rrbracket\in \mathfrak A(m,n)} C_\bullet(Z_\alpha),\quad\text{ where }\quad
C_\bullet(Z_\alpha)=\bigoplus_{j=0}^{dim(Z_\alpha)}\bigoplus_{i=1}^{r(\alpha,j)} \kk[j] \cdot t^{(j)}_{\alpha,i}
\end{equation}
The differential $d$ on $C_\bullet(Z_\alpha)$ maps $t^{(j)}_{\alpha,i}$ to a sum of basis elements of one lower degree, $d\big(t^{(j)}_{\alpha,i}\big)=\sum_{\ell=1}^{r(\alpha,j-1)} \delta^{(j\to j-1)}_{\alpha, i\to \ell} \cdot t^{(j-1)}_{\alpha,\ell}$. Here, by construction of $Z_\alpha$ as a contractible polytope, each coefficient $\delta^{(j\to j-1)}_{\alpha, i\to \ell}\in \{-1,0,+1\}$. The following two Claims \ref{CLM:linear-iso} and \ref{CLM:chain-map} state, that, the chain complex $\bD \mc V^{(d)}(m,n)$ is isomorphic to $C_\bullet(Z(m,n))$ up to shifts and signs. These claims, which are proved below, are enough to prove Theorem \ref{THM:V-is-Koszul}.
\begin{claim}\label{CLM:linear-iso}
There is a vector space isomorphism
$$
f:\bD \mc V^{(d)}(m,n)\stackrel\cong\to C_{\bullet}(Z(m,n))[d(m-1)]= \bigoplus_{\llbracket\alpha,\ooo,\iii\rrbracket\in \mathfrak A(m,n)} \bigoplus_{j=0}^{dim(Z_\alpha)}\bigoplus_{i=1}^{r(\alpha,j)} \kk[j+d(m-1)] \cdot t^{(j)}_{\alpha,i}.
$$
\end{claim}
Denote the corresponding basis of $\bD \mc V^{(d)}(m,n)$ by $\tilde{t}^{(j)}_{\alpha,i}:=f^{-1}(t^{(j)}_{\alpha,i})$.
\begin{claim}\label{CLM:chain-map}
The differential $d^\circ$ from \eqref{EQN:DP-complex} (see page \pageref{EQN:DP-complex}) on $\bD \mc V^{(d)}(m,n)$ under the above isomorphism has, up to signs, the same coefficients that $d$ has on $C_\bullet(Z(m,n))$, i.e. if we write
$$
d^\circ \big(\tilde{t}^{(j)}_{\alpha,i}\big)=\sum_{\ell=1}^{r(\alpha,j-1)} \tilde\delta^{(j\to j-1)}_{\alpha, i\to \ell} \cdot \tilde{t}^{(j-1)}_{\alpha,\ell}, \text{ then } \tilde\delta^{(j\to j-1)}_{\alpha, i\to \ell}=\pm \delta^{(j\to j-1)}_{\alpha, i\to \ell}.
$$
\end{claim}
In other words, the last Claim \ref{CLM:chain-map} says, that the coefficients of $d^\circ$ vanish precisely when the coefficients of $d$ vanish, but may differ from the coefficients of $d$ by a sign when they don't vanish. We first show how the above two claims imply Theorem \ref{THM:V-is-Koszul}.
\begin{proof}[Proof of Theorem \ref{THM:V-is-Koszul} using Claims \ref{CLM:linear-iso} and \ref{CLM:chain-map}]\label{PROOF:theorem35}
The linear isomorphism $f$ from Claim \ref{CLM:linear-iso} is by Claim \ref{CLM:chain-map}, in general, not necessarily a chain map, since the signs of $d$ and $d^\circ$ do not match. We can however define a new map $g:\bD \mc V^{(d)}(m,n)\stackrel\cong\to C_{\bullet}(Z(m,n))[d(m-1)]$ by fixing the signs of $f$, so that $g$ is an isomorphism of chain complex. This then in turn implies isomorphisms on homology, which implies the claim of the theorem, i.e.
\[
H^\bullet(\bD \mc V^{(d)}(m,n))\cong H_{\bullet}(Z(m,n))[d(m-1)]\cong \bigoplus_{\llbracket\alpha,\ooo,\iii\rrbracket\in \mathfrak A(m,n)} \kk[d(m-1)] \cong \mc V^{(-d)}(m,n).
\]

To define $g$, fix a component $Z_\alpha$, and recall that $Z_\alpha$ has only one top dimensional cell, indexed by $t^{(n+2m-4)}_{\alpha,1}$. For this cell we define $g(\tilde t^{(n+2m-4)}_{\alpha,1}):=f(\tilde t^{(n+2m-4)}_{\alpha,1})=t^{(n+2m-4)}_{\alpha,1}$. Now suppose $g(\tilde t^{(j)}_{\alpha,i})=\varepsilon^{(j)}_{\alpha,i}\cdot t^{(j)}_{\alpha,i}$, where $\varepsilon^{(j)}_{\alpha,i}\in \{+1,-1\}$, have been defined for all $j=n+2m-4,\dots, p$, in a way so that $g\big( d^\circ(\tilde t^{(j)}_{\alpha,i})\big)=d\big( g(\tilde t^{(j)}_{\alpha,i})\big)$ for any $j=n+2m-4,\dots, p+1$. We then define $g(\tilde t^{(p-1)}_{\alpha,\ell})$ as follows. Let $\tilde t^{(p)}_{\alpha,i}$ be the tree for any $p$ cell for which $\delta^{(p\to p-1)}_{\alpha, i\to \ell}\neq 0$. Then, define $g(\tilde t^{(p-1)}_{\alpha,\ell}):=\varepsilon^{(p-1)}_{\alpha,\ell}\cdot t^{(p-1)}_{\alpha,\ell}$, where
\begin{equation}\label{EQN:epsilon-p-alpha-l}
\varepsilon^{(p-1)}_{\alpha,\ell}:=\varepsilon^{(p)}_{\alpha,i}\cdot \delta^{(p\to p-1)}_{\alpha, i\to \ell} \cdot \tilde\delta^{(p\to p-1)}_{\alpha, i\to \ell}.
\end{equation}
We claim that this is well-defined (independent of the chosen $\tilde t^{(p)}_{\alpha,i}$), and that $g\big( d^\circ(\tilde t^{(p)}_{\alpha,i})\big)=d\big( g(\tilde t^{(p)}_{\alpha,i})\big)$.

To check that $g$ is well-defined, assume that there is another $\tilde t^{(p)}_{\alpha,i'}$ with 
$\delta^{(p\to p-1)}_{\alpha, i'\to \ell}\neq 0$. Since the non-zero faces in $Z_\alpha$ are precisely those obtained by edge expansion of a planar directed tree, we see that $t^{(p)}_{\alpha,i}$ and $t^{(p)}_{\alpha,i'}$ are obtained from $t^{(p-1)}_{\alpha,\ell}$ by collapsing an edge. Collapsing both these edges in $t^{(p-1)}_{\alpha,\ell}$ gives a new tree $t^{(p+1)}_{\alpha,k}$ which has $t^{(p)}_{\alpha,i}$ and $t^{(p)}_{\alpha,i'}$ as edge expansions, i.e. for which $\delta^{(p+1\to p)}_{\alpha, k\to i}\neq 0$ and $\delta^{(p+1\to p)}_{\alpha, k\to i'}\neq 0$, and thus also $\tilde \delta^{(p+1\to p)}_{\alpha, k\to i}\neq 0$ and $\tilde \delta^{(p+1\to p)}_{\alpha, k\to i'}\neq 0$. Since the only possibilities to obtain $t^{(p+1)}_{\alpha,k}$ from $t^{(p-1)}_{\alpha,\ell}$ via two edge collapses is by going through either $t^{(p)}_{\alpha,i}$ or $t^{(p)}_{\alpha,i'}$, we see from $d^2(t^{(p+1)}_{\alpha,k})=0$, that $\delta^{(p+1\to p)}_{\alpha, k\to i}\delta^{(p\to p-1)}_{\alpha, i\to \ell}=-\delta^{(p+1\to p)}_{\alpha, k\to i'}\delta^{(p\to p-1)}_{\alpha, i'\to \ell}$. Similarly, we obtain from $(d^\circ)^2(\tilde t^{(p+1)}_{\alpha,k})=0$, that $\tilde\delta^{(p+1\to p)}_{\alpha, k\to i}\tilde\delta^{(p\to p-1)}_{\alpha, i\to \ell}=-\tilde\delta^{(p+1\to p)}_{\alpha, k\to i'}\tilde\delta^{(p\to p-1)}_{\alpha, i'\to \ell}$. Thus, we obtain:
\begin{multline*}
\varepsilon^{(p)}_{\alpha,i}\cdot \delta^{(p\to p-1)}_{\alpha, i\to \ell} \cdot \tilde\delta^{(p\to p-1)}_{\alpha, i\to \ell}
=
\varepsilon^{(p+1)}_{\alpha,k}\cdot \delta^{(p+1\to p)}_{\alpha, k\to i} \cdot \tilde\delta^{(p+1\to p)}_{\alpha, k\to i}\cdot \delta^{(p\to p-1)}_{\alpha, i\to \ell} \cdot \tilde\delta^{(p\to p-1)}_{\alpha, i\to \ell}
\\
=
\varepsilon^{(p+1)}_{\alpha,k}\cdot \delta^{(p+1\to p)}_{\alpha, k\to i'} \cdot \tilde\delta^{(p+1\to p)}_{\alpha, k\to i'}\cdot \delta^{(p\to p-1)}_{\alpha, i'\to \ell} \cdot \tilde\delta^{(p\to p-1)}_{\alpha, i'\to \ell}
=
\varepsilon^{(p)}_{\alpha,i'}\cdot \delta^{(p\to p-1)}_{\alpha, i'\to \ell} \cdot \tilde\delta^{(p\to p-1)}_{\alpha, i'\to \ell}
\end{multline*}
This shows that $\varepsilon^{(p-1)}_{\alpha,\ell}$ in \eqref{EQN:epsilon-p-alpha-l} is well-defined.

It is now immediate to check that $g$ is a chain map:
\[
g\big( d^\circ(\tilde t^{(p)}_{\alpha,i})\big)
=\sum_{\ell=1}^{r} \tilde\delta^{(p\to p-1)}_{\alpha, i\to \ell} \cdot \varepsilon^{(p-1)}_{\alpha,\ell}\cdot t^{(p-1)}_{\alpha,\ell}
=d\big( g(\tilde t^{(p)}_{\alpha,i})\big)
\]
This completes the proof of Theorem \ref{THM:V-is-Koszul}.
\end{proof}
It remains to prove Claims \ref{CLM:linear-iso} and \ref{CLM:chain-map}.
\begin{proof}[Proof of Claim \ref{CLM:linear-iso}]
Recall from page \pageref{PAGE:DP(m,n)} that $\bD \mc V^{(d)}(m,n)\cong \bigoplus_{(m,n)\text{-tree$_0$ }T} (\mc V^{(d)})^*(T)\otimes \Det(T)$. Since we do not care about signs for either Claims \ref{CLM:linear-iso} and \ref{CLM:chain-map}, we will absorb those in vector space isomorphisms without further mention; in particular $\Det(T)\cong \kk[n+2m-3-|\Int(T)|]$. Now, denoting by $m_v=|\Out(v)|$ and $n_v=|\In(v)|$ the number of outgoing and incoming edges at a vertex $v$, respectively, we can write
\begin{align}\label{EQN:V(d)*(T)}
(\mc V^{(d)})^*(T)=&\bigotimes_{v\in \Int(T)} ( \mc V^{(d)})^*(\Out(v),\In(v))
\\
\cong & \bigotimes_{v\in \Int(T)} \bigg( \bigoplus\limits_{ \tiny\begin{array}{c}  f:\Out(v)\stackrel \sim\to \{1,\dots, m_v\}\\  g:\In(v)\stackrel \sim\to \{1,\dots, n_v\}\end{array}} \mc V^{(d)}(m_v,n_v)^*\bigg)_{(\Ss_{m_v},\Ss_{n_v})}\nonumber
\\
\cong & \bigotimes_{v\in \Int(T)} \bigg( \bigoplus\limits_{ \tiny\begin{array}{c}  f:\Out(v)\stackrel \sim\to \{1,\dots, m_v\}\\  g:\In(v)\stackrel \sim\to \{1,\dots, n_v\}\end{array}} \bigoplus_{\llbracket\alpha,\ooo,\iii\rrbracket\in \mathfrak A(m_v,n_v)} \kk[d(1-m_v)]^*\bigg)_{(\Ss_{m_v},\Ss_{n_v})}\nonumber
\\
\cong & \bigotimes_{v\in \Int(T)} \bigoplus_{\llbracket\alpha,\ooo,\iii\rrbracket\in \mathfrak A(\Out(v),\In(v))} \kk[d(m_v-1)],\nonumber
\end{align}
where $\mathfrak A(X,Y)$\label{PAGE:A(X,Y)} consists of equivalence classes of triples $(\alpha,\ooo,\iii)$, $\alpha:\Z_{n+m}\to \{\oo,\ii\}$, $\ooo:X\to \Z_{n+m}$, $\iii:Y\to \Z_{n+m}$ with $Image(\ooo)=\alpha^{-1}(\oo)$ and $Image(\iii)=\alpha^{-1}(\ii)$, modulo the relation generated by $(\alpha,\ooo,\iii)\sim (\alpha\circ \tau,\tau^{-1}\circ\ooo,\tau^{-1}\circ\iii)$; cf. Definition \ref{DEF:A(m,n)-gamma}. Since the sum $\sum_{v\in \Int(T)} (m_v-1)=m-1$ is the total number $m$ of roots of $T$ minus $1$, we see that 
\[
(\mc V^{(d)})^*(T)\cong \bigoplus_{\tiny\begin{array}{c}\forall v\in\Int(T)\text{ choose}\\\llbracket\alpha,\ooo,\iii\rrbracket\in \mathfrak A(\Out(v),\In(v))\end{array}} \kk[d(m-1)].
\]
Thus, we obtain that
\begin{equation*}\label{EQN:DV=T-cyclic-v}
\bD \mc V^{(d)}(m,n) \cong 
 \bigoplus_{(m,n)\text{-tree$_0$ }T} \bigoplus_{\tiny\begin{array}{c}\forall v\in\Int(T)\text{ choose}\\\llbracket\alpha,\ooo,\iii\rrbracket\in \mathfrak A(\Out(v),\In(v))\end{array}} \kk[d(m-1)+(n+2m-3-\Int(T))].
\end{equation*}
Now, note that the data of an $(m,n)$-tree$_0$ $T$ with a choice of cyclic order $\llbracket\alpha,\ooo,\iii\rrbracket\in \mathfrak A(\Out(v),\In(v))$ at each of its internal vertices $v$ is equivalent to writing $T$ as a directed planar tree by assembling the edges at each vertex according to its cyclic structure; see the planar tree in the center of Figure \ref{FIG:(m,n)-tree0-vs-cells-of-Z}. Here, each leaf has a unique label from $\{1,\dots, n\}$, while each root has a unique label from $\{1,\dots, m\}$.
\begin{figure}[h]
\[
\hspace{-.2cm} \includegraphics[scale=.95]{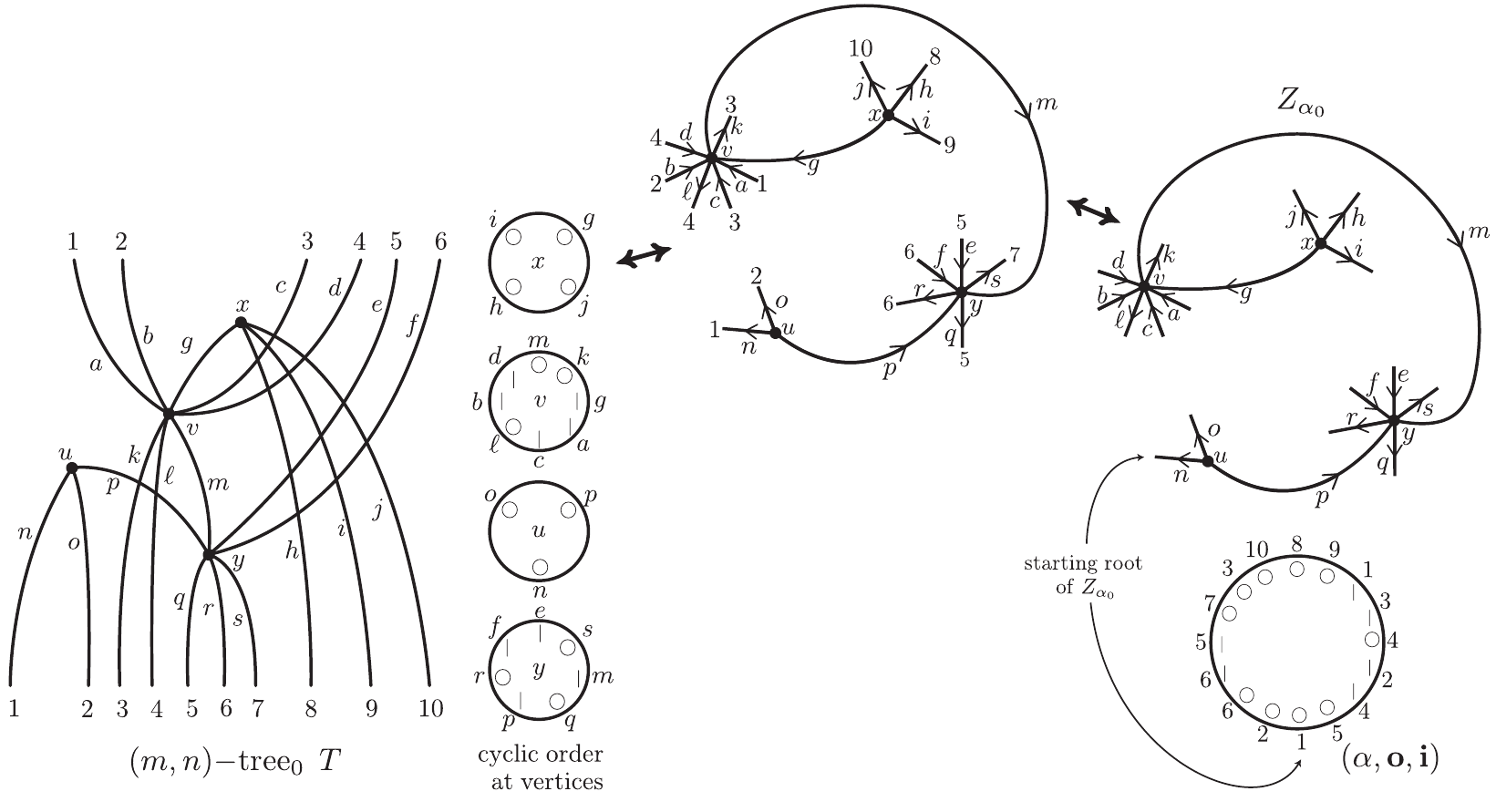}
\]
\caption{Correspondence between an $(m,n)$-tree$_0$ with cyclic order at their internal vertices, and the class $\llbracket\alpha,\ooo,\iii\rrbracket$ together with a cell of $Z_{\alpha_0}$. The edges are denoted by $a$, $b$, $c$, $\dots$, $s$, the internal vertices are denoted $u$, $v$, $x$, $y$.}\label{FIG:(m,n)-tree0-vs-cells-of-Z}
\end{figure}

Moreover, this planar structure induces a cyclic order on the set of all roots and leaves of the $(m,n)$-tree$_0$ $T$ (which is determined only up to cyclic order as it depends on the starting root or leaf), which is precisely the data of a class $\llbracket\alpha,\ooo,\iii\rrbracket$. With this, the directed planar tree obtained by forgetting these labels at the roots and leaves, is precisely a tree $t_\alpha$ indexing a cell of $Z_\alpha$; see the right side of Figure \ref{FIG:(m,n)-tree0-vs-cells-of-Z}. Note that although $\alpha$ is defined only up to cyclic order, it does not matter which $Z_\alpha$ we choose, since $Z_{\alpha}\simeq Z_{\alpha\circ \tau}$ are isomorphic cell complexes (cf. \cite[Corollary 3.11(2)]{PT}). However, in order to pick a well-defined $Z_{\alpha_0}$, we will choose the unique $\alpha_0=\alpha\circ \tau^k$, which has the ``starting root'' at $0\in \Z_{n+m}$, i.e. at $0$ it is outgoing, $\alpha_0(0)=\oo$, with outgoing label ``$1$'', $\ooo(1)=0$; see the right side of Figure \ref{FIG:(m,n)-tree0-vs-cells-of-Z}.

Thus, we claim that there is an isomorphism of vector spaces:
\begin{equation}\label{EQN:DV=[aoi]-Za}
\bD \mc V^{(d)}(m,n) \cong \bigoplus_{\llbracket\alpha,\ooo,\iii\rrbracket\in \mathfrak A(m,n)} \,\,\,\bigoplus_{\text{cells $c$ of }Z_{\alpha_0}} \kk[d(m-1)+dim(\text{cell }c)].
\end{equation}
The stated data of a class $\llbracket\alpha,\ooo,\iii\rrbracket\in \mathfrak A(m,n)$ with a cell of $Z_{\alpha_0}$ indexed by $t_\alpha$ is equivalent to the staring data of an $(m,n)$-tree$_0$ $T$ with a cyclic order at its vertices $\llbracket\alpha,\ooo,\iii\rrbracket\in \mathfrak A(\Out(v),\In(v))$, since we can give an inverse map from the right side of Figure \ref{FIG:(m,n)-tree0-vs-cells-of-Z} to the left side: place labels of $\ooo$ and $\iii$ around $t_\alpha$, forget the planar structure and call the resulting $(m,n)$-tree$_0$ $T$, and, in addition, record the individual cyclic orders at each interior vertex $v$ of $T$. The maps from right to left, and from left to right of Figure \ref{FIG:(m,n)-tree0-vs-cells-of-Z}, as described above, are inverses of each other.

Note furthermore, that the dimension of a cell indexed by a tree $t_\alpha$ of $Z_\alpha$ is precisely $n+2m-3-\#($internal vertices of $t_\alpha)$. Thus, \eqref{EQN:DV=[aoi]-Za} establishes an isomorphism as stated in Claim \ref{CLM:linear-iso}.
\end{proof}
\begin{proof}[Proof of Claim \ref{CLM:chain-map}]
We need to compare the differentials $d^\circ$ on $\bD \mc V^{(d)}(m,n)$ with $d$ on $C_\bullet(Z(m,n))$ under the isomorphism from Claim \ref{CLM:linear-iso}. The main (and technical) part of this proof is \eqref{EQN:dcirc-on-V(d)}, which states that the composition in $\mc V^{(d)}$ is given by combining cyclic orders at a vertex along their common edge to give the resulting cyclic order of the composition.

To state \eqref{EQN:dcirc-on-V(d)} precisely, we recall the already used identification of a basis of $\mc V^{(d)}(X,Y)$ for \emph{sets} $X$ and $Y$ with cyclic orders on $X\sqcup Y$:
\begin{multline}\label{EQN:V(d)(XY)-iso}
 \mc V^{(d)}(X,Y)
\cong
\bigg( \bigoplus\limits_{ \tiny\begin{array}{c}  f:X\stackrel \sim\to \{1,\dots, m\}\\  g:Y\stackrel \sim\to \{1,\dots, n\}\end{array}} \mc V^{(d)}(m,n)\bigg)_{(\Ss_{m},\Ss_{n})}\\
\cong
 \bigg( \bigoplus\limits_{ \tiny\begin{array}{c}  f:X\stackrel \sim\to \{1,\dots, m\}\\  g:Y\stackrel \sim\to \{1,\dots, n\}\end{array}} \bigoplus_{\llbracket\alpha,\ooo,\iii\rrbracket\in \mathfrak A(m,n)} \kk[d(1-m)]\bigg)_{(\Ss_{m},\Ss_{n})}
\cong
\bigoplus_{\llbracket\alpha,\ooo,\iii\rrbracket\in \mathfrak A(X,Y)} \kk[d(1-m)].
\end{multline}
Thus, a basis element $\rho\in \mc V^{(d)}(X,Y)$ is uniquely given by a class $\llbracket\alpha,\ooo,\iii\rrbracket\in \mathfrak A(X,Y)$, which can be represented by a bijection $\rho_{\ooo\iii}:=\ooo\sqcup \iii:X\sqcup Y\stackrel \sim\to \Z_{n+m}$ which itself is determined only up to cyclic rotation on $\Z_{n+m}$; cf. page \pageref{PAGE:A(X,Y)}. Now, the composition in $\mc V^{(d)}$ applied to sets $X$, $Y$, $\tilde X$, $\tilde Y$ along an edge connecting $\tilde x\in \tilde X$ with $y \in Y$ is a map $_y\circ _{\tilde x}:\mc V^{(d)}(X,Y)\otimes \mc V^{(d)}(\tilde X,\tilde Y)\to \mc V^{(d)}(X\sqcup (\tilde X-\{\tilde x\}),(Y-\{y\})\sqcup \tilde Y)$. For $\rho \in \mc V^{(d)}(X,Y)$ given by $\rho_{\ooo\iii}$ with $\rho_{\ooo\iii}(y)=n+m-1$ and $\tilde \rho \in \mc V^{(d)}(\tilde X,\tilde Y)$ given by $\tilde \rho_{\ooo\iii}$ with $\tilde \rho_{\ooo\iii}(\tilde x)=0$, we claim that $\rho_y\circ _{\tilde x}\tilde \rho$ is the basis element with cyclic order given by ``cutting'' the cyclic orders $\rho_{\ooo\iii}$ and $\tilde \rho_{\ooo\iii}$ at $y$ and $\tilde x$, respectively, and combining them along their boundaries:
\begin{align}\label{EQN:dcirc-on-V(d)}
& (\rho_y\circ _{\tilde x}\tilde \rho)_{\ooo\iii}:X\sqcup (\tilde X-\{\tilde x\})\sqcup (Y-\{y\})\sqcup \tilde Y\to \Z_{m+n+\tilde m+\tilde n-2},\\
& (\rho_y\circ _{\tilde x}\tilde \rho)^{-1}_{\ooo\iii}(j)=\left\{\begin{array}{ll}\rho^{-1}_{\ooo\iii}(j), &\text{for }0\leq j \leq m+n-2 \\ \tilde \rho_{\ooo\iii}^{-1}(j-m-n+2), &\text{for }m+n-1\leq j \leq m+n+\tilde m+\tilde n -3 \end{array}\right. \nonumber\\
& \includegraphics[scale=.8]{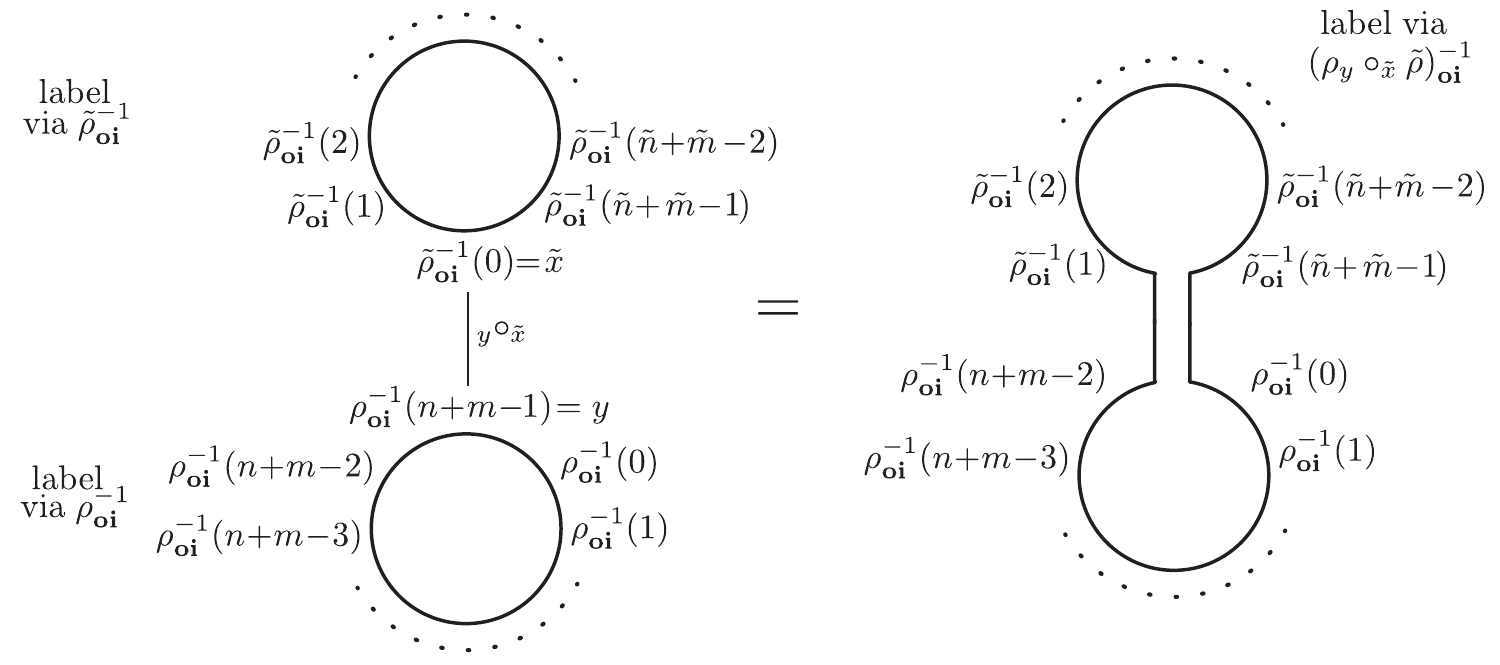}
\nonumber
\end{align}
Once we have established \eqref{EQN:dcirc-on-V(d)}, the result of Claim \ref{CLM:chain-map} can be seen as follows. Dualizing the linear map $_y\circ _{\tilde x}$ shows that $(_y\circ _{\tilde x})^*(\rho^*)$ for some $\rho^*\in (\mc V^{(d)})^*(X,Y)$ for some cyclic order $\rho_{\ooo\iii}$ is given by a sum $(_y\circ _{\tilde x})^*(\rho^*)=\sum (\rho')^*\otimes (\rho'')^*$ of all possible ways of breaking $\rho_{\ooo\iii}$ into two cyclic orders $\rho'_{\ooo\iii}$ and $\rho''_{\ooo\iii}$ so that their composition yields the given order $\rho_{\ooo\iii}$ as in \eqref{EQN:dcirc-on-V(d)}. Using this dual of the composition under the isomorphism \eqref{EQN:V(d)*(T)} shows, up to sign, that $d^\circ$ (which is essentially the dual of the compositions $_y\circ _{\tilde x}$, cf. \eqref{EQN:DP-complex} on page \pageref{EQN:DP-complex}) splits a vertex $v$ of an $(m,n)$-tree$_0$ $T$ into two ``allowable'' vertices connected via one directed edge. (Here, ``allowable'' means that the new vertices have the right number of outputs and inputs, as specified by $\JJ$.) The total cyclic order of the combined vertices after connecting them via their common edge coincides with the original one. Thus, this shows that, again up to sign, going from the left side of Figure \ref{FIG:(m,n)-tree0-vs-cells-of-Z} to the right side of Figure \ref{FIG:(m,n)-tree0-vs-cells-of-Z} keeps the global class $\llbracket\alpha,\ooo,\iii\rrbracket\in \mathfrak A(m,n)$ intact, while a tree $t_\alpha$ indexing a cell of $Z_{\alpha_0}$ gets mapped to an ``allowable'' $1$-edge expansion of $t_\alpha$. These are precisely the trees indexing the boundary components of the cell in $Z_{\alpha_0}$ indexed by $t_\alpha$. Thus, this proves Claim \ref{CLM:chain-map}.

It remains to prove \eqref{EQN:dcirc-on-V(d)}. We first describe how the $_i\circ _j$-operations of a dioperad $\mc P$ are extended to an operation on $\mc P(X,Y)= \Bigg( \bigoplus\limits_{ \tiny\begin{array}{c} f:X\stackrel \sim\to \{1,\dots, m\}\\  g:Y\stackrel \sim\to \{1,\dots, n\}\end{array}} \mc P(m,n)\Bigg)_{(\Ss_m,\Ss_n)}$ for \emph{sets} $X$ and $Y$; cf. \cite[(1.32)]{MSS}. More precisely, consider sets $X$, $Y$, $\tilde X$, $\tilde Y$ with $X, Y, \tilde X \neq \varnothing$, and a choice of $\tilde x\in \tilde X$ and $y\in Y$, and denote by $m=|X|$, $n=|Y|$, $\tilde m=|\tilde X|$, and $\tilde n=|\tilde Y|$. Then, there are induced maps
\begin{align}\label{EQN:y-circ-x-on-P}
_y\circ_{\tilde x}:& \mc P(X,Y)&\otimes & \mc P(\tilde X,\tilde Y)&\to & \mc P(X\sqcup (\tilde X-\{\tilde x\}), (Y-\{y\})\sqcup \tilde Y)\\
& \left[(\rho)_{\tiny\begin{array}{c} f:X\stackrel \sim\to \{1,\dots, m\}\\  g:Y\stackrel \sim\to \{1,\dots, n\}\end{array}}\right]&\otimes & \left[(\tilde \rho)_{\tiny\begin{array}{c} \tilde f:\tilde X\stackrel \sim\to \{1,\dots, \tilde m\}\\  \tilde g:\tilde Y\stackrel \sim\to \{1,\dots, \tilde n\}\end{array}}\right]&\mapsto & \left[(\rho \,_{g(y)}\circ_{\tilde f(\tilde x)} \tilde \rho)_{\tiny\begin{array}{c} f^\circ\\  g^\circ\end{array}}\right],\nonumber
\end{align}
where $f^\circ:X\sqcup (\tilde X-\{\tilde x\})\to \{1,\dots, m+\tilde m-1\}$ is $f^\circ(w):=f(w)+\tilde f(\tilde x)-1$ for $w\in X$, and $f^\circ(w)=\left\{\begin{array}{ll}\tilde f(w), &\text{for }\tilde f(w)<\tilde f(\tilde x)\\ \tilde f(w)+m-1, &\text{for }\tilde f(w)>\tilde f(\tilde x)\end{array}\right.$ for $w\in \tilde X-\{\tilde x\}$, and $g^\circ:(Y-\{y\})\sqcup \tilde Y\to \{1,\dots, n+\tilde n-1\}$ is given by $g^\circ (w)=\left\{\begin{array}{ll}g(w), &\text{for }g(w)<g(y)\\ g(w)+\tilde n-1, &\text{for }g(w)>g(y)\end{array}\right.$ for $w\in Y-\{y\}$, and $g^\circ(w)=\tilde g(w)+g(y)-1$ for $w\in \tilde Y$; see Figure \ref{FIG:composition-P(XY)}.
\begin{figure}[h]
\[
\includegraphics[scale=.9]{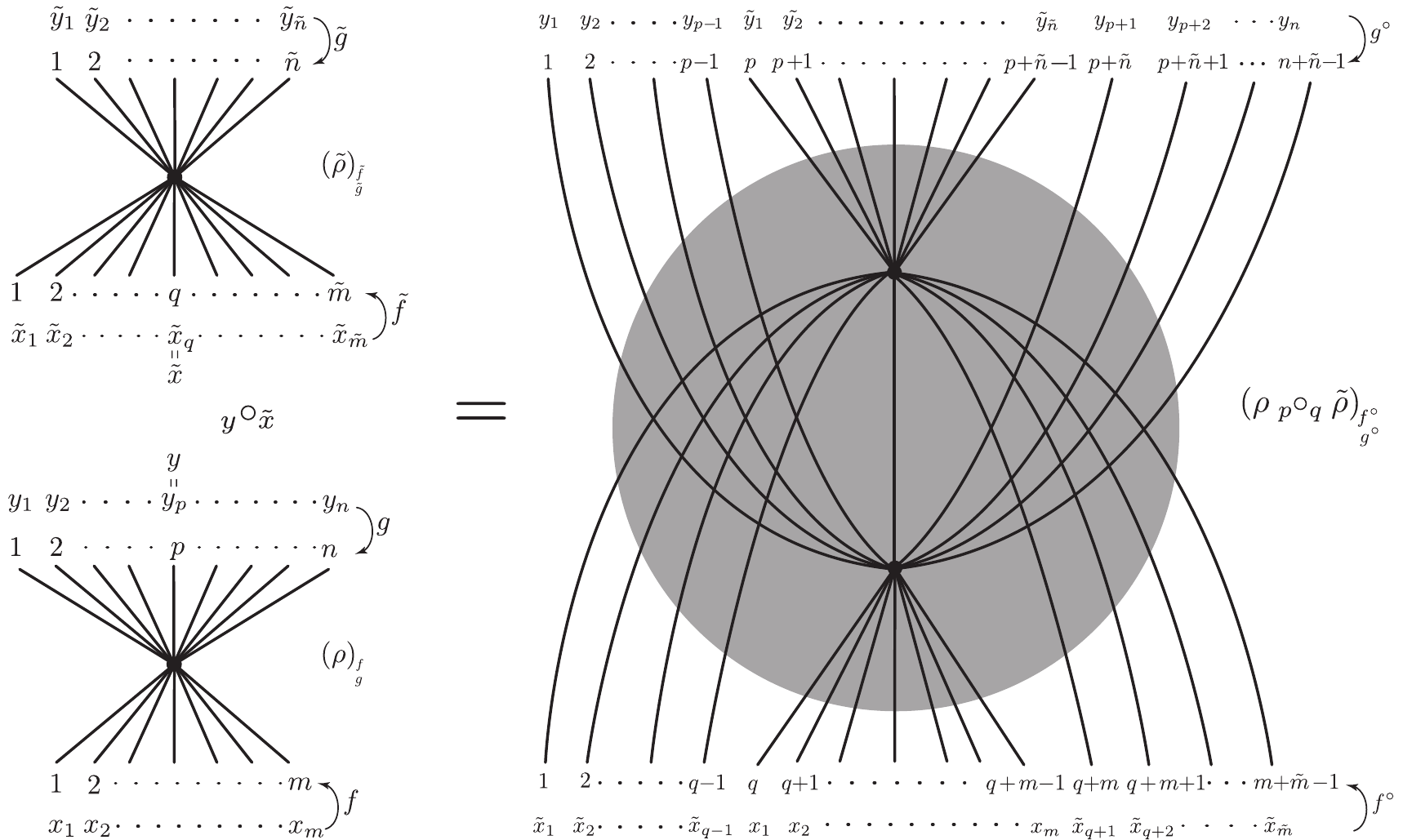}
\]
\caption{Composition for a dioperad $\mc P$ with inputs and outputs labeled by sets}\label{FIG:composition-P(XY)}
\end{figure}
Using the equivariance property (Definition \ref{DEF:dioperad0}(c)), one checks that this map is well-defined, i.e. independent of the chosen representatives $\rho$, $f$, $g$, $\tilde \rho$, $\tilde f$, $\tilde g$, just as in \cite[page 66, (1.32)]{MSS}.

Next we identify the above composition operations $_y\circ_{\tilde x}$ for the dioperad $\mc V^{(d)}$ under the isomorphism $\mc V^{(d)}(X,Y)\cong \bigoplus_{\llbracket\alpha,\ooo,\iii\rrbracket\in \mathfrak A(X,Y)} \kk[d(1-m)]$ from \eqref{EQN:V(d)(XY)-iso}. Starting from basis elements $\rho\in \mc V^{(d)}(X,Y) $ and $\tilde \rho \in \mc V^{(d)}(\tilde X,\tilde Y)$, we want to compute $\rho_y\circ_{\tilde x} \tilde \rho$ for some $y\in Y$ and $\tilde x \in \tilde X$. As described above, $\rho$ and $\tilde \rho$ have associated cyclic orders $\rho_{\ooo\iii}:X\sqcup Y\to \Z_{m+n}$ and $\tilde \rho_{\ooo\iii}:\tilde X\sqcup \tilde Y\to \Z_{\tilde m+\tilde n}$ under \eqref{EQN:V(d)(XY)-iso}. Since $\rho_{\ooo\iii}$ and $\tilde \rho_{\ooo\iii}$ are only determined up to cyclic rotation, we may assume, without loss of generality, that $\tilde\rho^{-1}_{\ooo\iii}(0)=\tilde x$, while we assume that $\rho^{-1}_{\ooo\iii}(0)=x$ for some $x\in X$. Using these choices for $\rho_{\ooo\iii}$ and $\tilde \rho_{\ooo\iii}$, there are $f:X\to \{1,\dots, m\}$, $g:Y\to \{1,\dots, n\}$, $\alpha\in \mathfrak A(m,n)$ and similarly $\tilde f:\tilde X\to \{1,\dots, \tilde m\}$, $\tilde g:\tilde Y\to \{1,\dots, \tilde n\}$, $\tilde \alpha\in \mathfrak A(\tilde m,\tilde n)$ so that $\rho_{\ooo\iii}=(\ooo_\alpha\circ f)\sqcup (\iii_\alpha\circ g)$ and $\tilde \rho_{\ooo\iii}=(\ooo_{\tilde \alpha}\circ \tilde f)\sqcup (\iii_{\tilde \alpha}\circ \tilde g)$.
\[
\includegraphics[scale=.75]{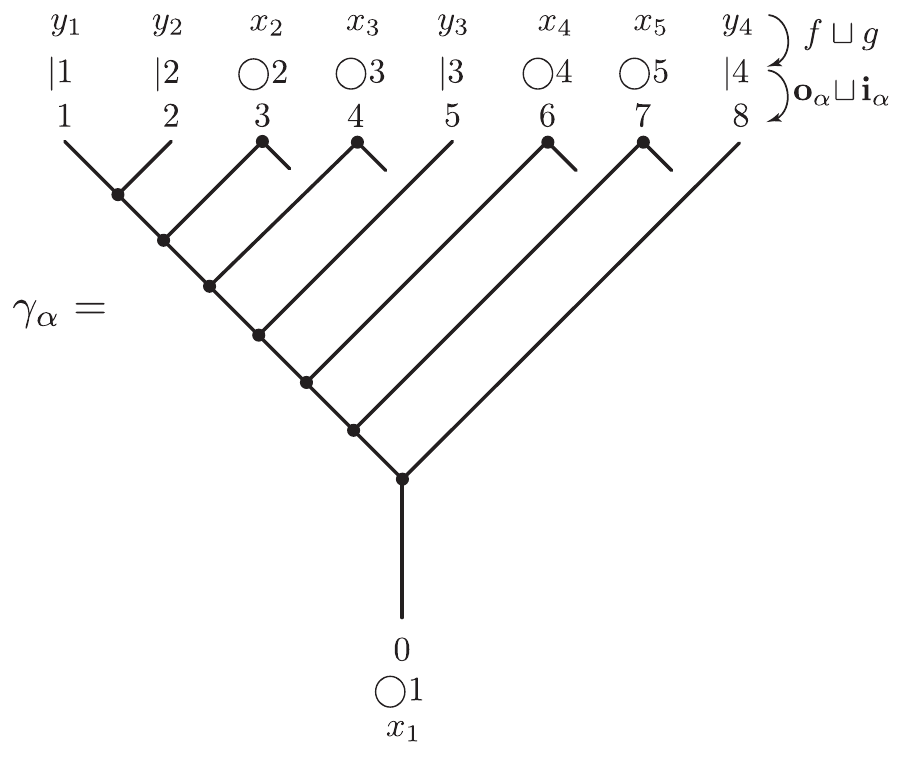}
\]
Thus, $\rho=(\gamma_\alpha)_{\tiny\begin{array}{c} f\\  g\end{array}}$ and $\tilde \rho=(\gamma_{\tilde \alpha})_{\tiny\begin{array}{c} \tilde f\\ \tilde g\end{array}}$ are just $\gamma_\alpha$ and $\gamma_{\tilde \alpha}$ with the particular labels from $f$, $g$, $\tilde f$, and $\tilde g$. Thus, by the description of the composition \eqref{EQN:y-circ-x-on-P} for sets, we need to calculate $(\gamma_\alpha \,_{g(y)}\circ_{\tilde f(\tilde x)}\gamma_{\tilde \alpha}))_{\tiny\begin{array}{c} f^\circ\\  g^\circ\end{array}}$. Since by assumption $\tilde f(\tilde x)=\ooo_{\tilde \alpha}^{-1}(\tilde\rho_{\ooo\iii}(\tilde x))=\ooo_{\tilde \alpha}^{-1}(0)=1$, it follows that we compose $\gamma_{\tilde \alpha}$ at the special first output of $\tilde \alpha$. For $\gamma_\alpha$, the position of the input composition is denoted by $p:=g(y)$, and assume that this $p$th input position occurs exactly between the $r$th and $(r+1)$st output position in the cyclic order of $\alpha$; see left side of \eqref{EQN:composing-gamma-alpha}. We obtain the following composition for $\gamma_\alpha \,_{g(y)}\circ_{\tilde f(\tilde x)}\gamma_{\tilde \alpha}=\gamma_\alpha \,_p\circ_1 \gamma_{\tilde \alpha}$:
\begin{equation}\label{EQN:composing-gamma-alpha}
\includegraphics[scale=.85]{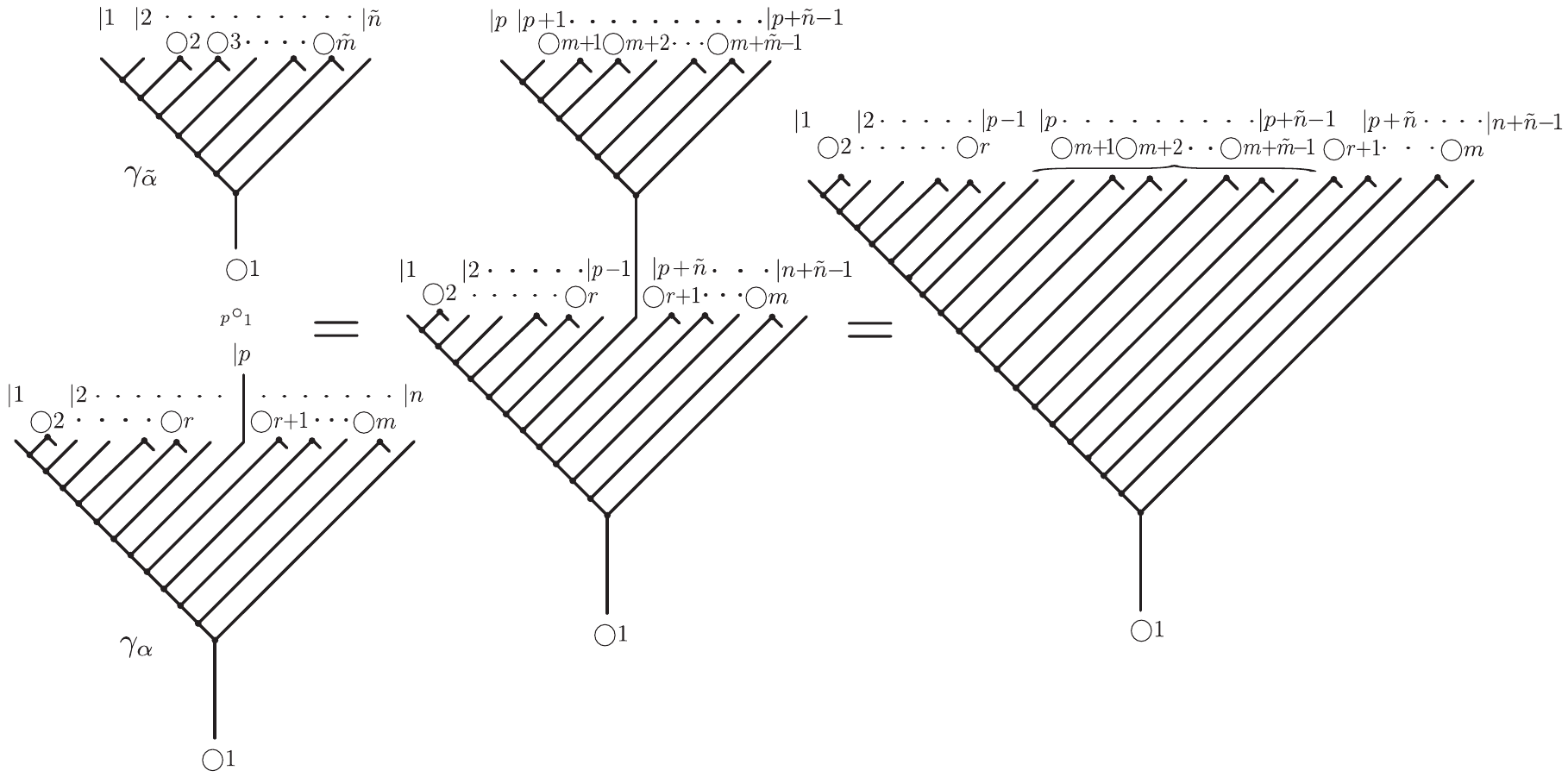}
\end{equation}
Notice, in particular, the new labeling of the outputs in $\gamma_\alpha \,_p\circ_1 \gamma_{\tilde \alpha}$ due to the fact that we composed at the first spot in $\gamma_{\tilde \alpha}$, so that all output labels of $\gamma_{\alpha}$ come before those of $\gamma_{\tilde \alpha}$, while the input labels arrange according to the cyclic order. Fortunately, the labeling of $f^\circ$ and $g^\circ$ as depicted in Figure \ref{FIG:composition-P(XY)} recover the wanted cyclic order on $\rho_y\circ_{\tilde x}\tilde \rho$:
\begin{equation}\label{EQN:composed-gamma-alpha}
\includegraphics[scale=.9]{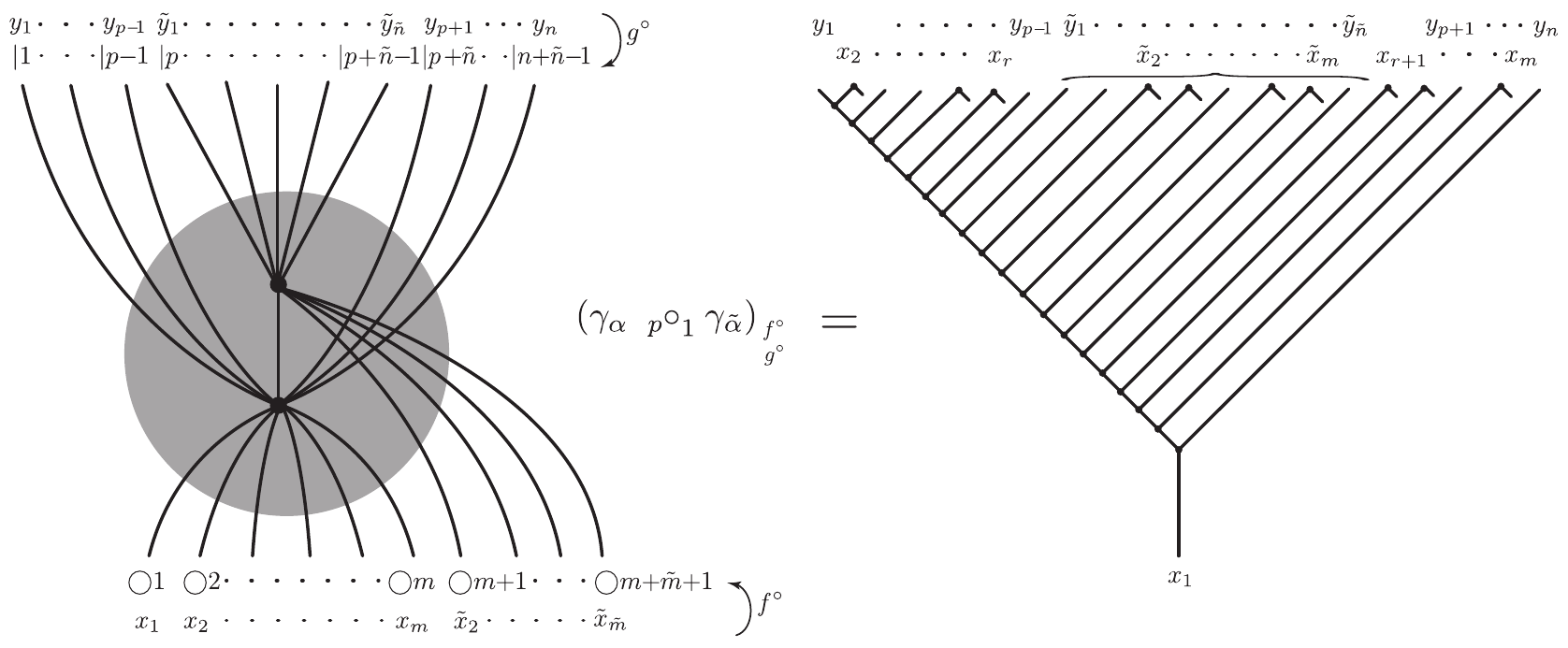}
\end{equation}
Comparing the cyclic orders (of $X$, $Y$, $\tilde X$, and $\tilde Y$ via $f$, $g$, $\tilde f$, and $\tilde g$) applied to $\gamma_\alpha$ and $\gamma_{\tilde \alpha}$ on the left hand side of \eqref{EQN:composing-gamma-alpha} with the cyclic order on the right of \eqref{EQN:composed-gamma-alpha} shows, that the composition $\mc V^{(d)}$ ``glues'' the cyclic orders of $\rho$ and $\tilde \rho$ as described in \eqref{EQN:dcirc-on-V(d)}. This completes the proof of \eqref{EQN:dcirc-on-V(d)}.
\end{proof}

\subsection{$\mc V_\infty^{(d)}$-algebras}

We end this section with a description of a $\mc V^{(d)}_\infty$-algebra.
\begin{defn}\label{DEF:V-infty-algebra}
If $\mc P$ is quadratic and Koszul, then define $\mc P_\infty:=\bD \mc P^!$. An algebra over the $\mc P_\infty$ dioperad is called a $\mc P_\infty$-algebra.
\end{defn}
In the case of $\mc P=\mc V^{(d)}$, we thus get $\mc V^{(d)}_\infty=\bD (\mc V^{(d)})^!=\bD \mc V^{(-d)}$. Since all operations of $\bD \mc V^{(-d)}$ are generated under composition by the corollas, one for each class $\llbracket\alpha,\ooo,\iii\rrbracket\in \mathfrak A(m,n)$, it follows that the data of a $\mc V^{(d)}_\infty$-algebra is given by a dg-vector space $(A,d_A)$ together with elements
$\mathfrak v_{\llbracket\alpha,\ooo,\iii\rrbracket}\in Hom(A^{\otimes n},A^{\otimes m})\cong A^{\otimes m}\otimes(A^*)^{\otimes n}$.
Choosing a representative $(\alpha,\ooo,\iii)$ determines a canonical isomorphism $\varrho_{(\alpha,\ooo,\iii)}:A^{\otimes m}\otimes(A^*)^{\otimes n}  \to A_{\alpha(0)}\otimes A_{\alpha(1)}\otimes \dots \otimes A_{\alpha(n+m-1)}$, where we set $A_{\ii}:=A^*$ and $A_{\oo}:=A$, and the isomorphism $\varrho_{(\alpha,\ooo,\iii)}$ rearranges the tensor factors using the usual Koszul sign rule according to the representative $(\alpha,\ooo,\iii)$ of $\llbracket\alpha,\ooo,\iii\rrbracket$. 
We define $\mathfrak{v}_{\alpha}$ as the image of $\mathfrak v_{\llbracket\alpha,\ooo,\iii\rrbracket}$ under this isomorphism:
\begin{equation}\label{EQU:v-alphas-domain}
\mathfrak v_{\alpha} = \varrho_{(\alpha,\ooo,\iii)}(\mathfrak v_{\llbracket\alpha,\ooo,\iii\rrbracket}) \in A_{\alpha(0)}\otimes A_{\alpha(1)}\otimes \dots \otimes A_{\alpha(n+m-1)}
\end{equation}
The elements $\mathfrak v_{\alpha}$ satisfy the following conditions:
\begin{itemize}
\item
\emph{Degree:} If $\alpha$ has $n$ input labels ``$\ii$'' and $m$ output labels ``$\oo$'', then $\mathfrak v_\alpha$ is of degree 
\begin{equation}\label{EQN:|v-alpha|}
 |\mathfrak v_{\alpha}|=4-n-2m+d\cdot (m-1). 
\end{equation}
\item
\emph{Symmetry condition:} The cyclic rotation $\tau$ (see Definition \ref{DEF:A(m,n)-gamma}) acts on the $\mathfrak v_\alpha$s via the usual Koszul sign rule: \begin{equation}
 \mathfrak v_{\alpha\circ \tau}=\tau(\mathfrak v_\alpha) 
\end{equation}
Here, the cyclic rotation $\tau:\Z_{n+m}\to \Z_{n+m}$ acts on the right by rotating tensor factors.
\item\label{EQU:v-a-boundary-condition}
\emph{Boundary condition:} If $\alpha$ has $n$ input labels ``$\ii$'' and $m$ output labels ``$\oo$'', then
\begin{equation}\label{EQN:dA(v-alpha)}
d_A(\mathfrak v_\alpha)=\sum_{\tiny\begin{array}{c}(m,n)\text{-tree$_0$ } T \\ \text{with one internal edge } e\\ \text{so that }\alpha'\circ_e \alpha''=\alpha \end{array}} (-1)^\epsilon \cdot \mathfrak v_{\alpha'}\circ \mathfrak v_{\alpha''},
\end{equation}
where ``$\circ$'' is the operation that contracts output of $\mathfrak v_{\alpha''}$ with the input of $\mathfrak v_{\alpha'}$ determined by the interior edge of $T$.
\end{itemize} 
To explain the sign $\epsilon$ appearing in \eqref{EQN:dA(v-alpha)}, we explicitly state  \eqref{EQN:dA(v-alpha)} in some examples.
\begin{ex}
Let $(A,d_A)$ be a $\mc V_\infty^{(d)}$-algebra, i.e. the elements $\mathfrak v_\alpha$ from \eqref{EQU:v-alphas-domain} are the images of a tree$_0$ with one internal vertex (corolla) labeled by $\gamma^*_\alpha$ under a dioperad map $f:\mc V_\infty^{(d)}\to \mc End_A$ composed with the appropriate $\varrho_{(\alpha,\ooo,\iii)}$. (The degree of $\mathfrak v_\alpha$ for an $\alpha$ with $n$ inputs and $m$ outputs is $|\mathfrak v_\alpha|=4-n-2m+d\cdot (m-1)$.) We calculate $d^\circ$ on such a corolla labeled by $\gamma^*_\alpha$ in $\bD(\mathcal V^{(-d)})(m,n)\cong \bigoplus_{(m,n)\text{-tree}_0\text{s } T} (\mc V^{(-d)})^*(T)\otimes  \bigwedge\,^{|\Int(T)|}(\kk[-1])^{\Int(T)} \otimes \Omega(m,n)$ for a summand whose vertices are labeled by $\gamma^*_{\alpha_1}$ and $\gamma^*_{\alpha_2}$ as follows:
\[
\includegraphics[scale=.89]{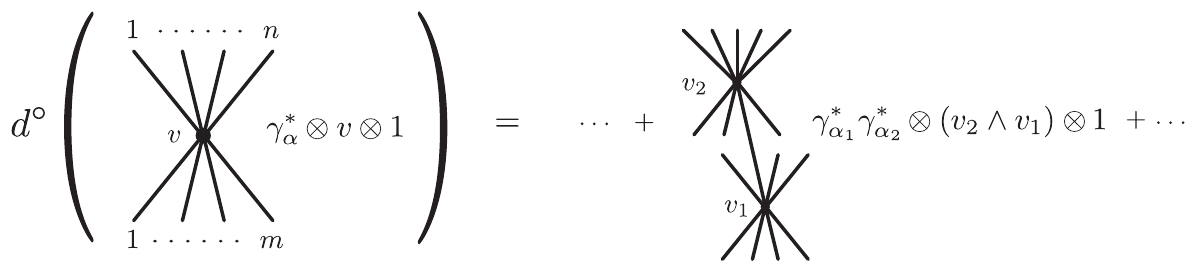}
\]
Here, the degrees of the elements stated above are $|v|=1$, $|1|=3-n-2m$ since $1=1_{\Omega(n,m)} \in \Omega(n,m)$, and $|\gamma^*_{\alpha}|=d(m-1)$ since $\gamma^*_{\alpha}\in (\mc V^{(-d)})^*(n,m)$ includes exactly $m-1$ co-inner products (see Figure \ref{FIG:gamma-alpha}). On the other hand, calculating the corresponding composition of corollas labeled by $\gamma^*_{\alpha_1}$ and $\gamma^*_{\alpha_2}$ gives
\[
\includegraphics[scale=.89]{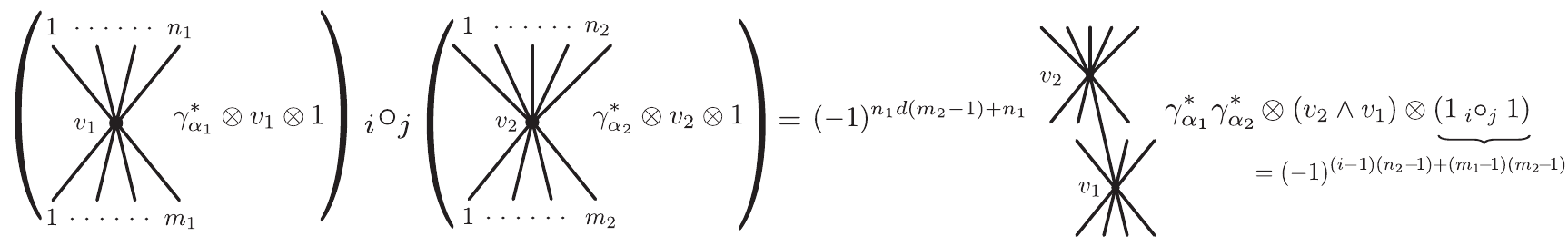}
\]
Here the degrees are $|\gamma^*_{\alpha_1}|=d\cdot (m_1-1)$, $|\gamma^*_{\alpha_2}|=d\cdot (m_2-1)$, $|v_1|=1$, $|v_2|=1$, $|1_{\Omega(m_1,n_1)}|=3-n_1-2m_1$, and $|1_{\Omega(m_2,n_2)}|=3-n_2-2m_2$. The overall sign on the right hand side comes from the usual Koszul sign rule, since both $\gamma^*_{\alpha_2}$ and $v_2$ move over $v_1$ and $1_{\Omega(m_1,n_1)}$. We thus get an overall sign of
\begin{equation}\label{EQU:signs-in-dA}
d_A(\mathfrak v_\alpha)=\dots +(-1)^{dn_1(m_2-1)+n_1+(i-1)(n_2-1)+(m_1-1)(m_2-1)} (\pi,\sigma).\mathfrak v_{\alpha_1}{}_i\circ_j \mathfrak v_{\alpha_2}+\dots.
\end{equation}

We evaluate this in three cases. Denote in particular, $\mathfrak m_1:=d_A$, $\mathfrak m_n:=\mathfrak v_{(\oo\ii\dots\ii)}\in A\otimes (A^*)^{\otimes n}$ for $n\geq 2$, $c:=\mathfrak v_{(\oo\oo)}\in A\otimes A$, $\mathfrak h:=\mathfrak v_{(\oo\oo\ii)}\in A\otimes A\otimes A^*$, and $\mathfrak l:=\mathfrak v_{(\oo\oo\oo)}\in A\otimes A\otimes A$.
\begin{enumerate}
\item
If $\alpha=(\oo\ii\dots \ii)$ with $n$ inputs, then necessarily $\alpha_1=(\oo\ii\dots \ii)$ with $n_1$ inputs, and $\alpha_2=(\oo\ii\dots \ii)$ with $n_2$ inputs, where $n_1+n_2-1=n$, and $m_1=m_2=j=1$. The sign in \eqref{EQU:signs-in-dA} thus becomes $(-1)^{n_1+(i-1)(n_2-1)}$. Applying this to elements $a_1,\dots, a_n\in A$, we obtain:
\begin{multline*}
\quad\quad [d_A,\mathfrak m_n](a_1,\dots, a_n) \\
=\sum_{n_1+n_2=n+1}\sum_i (-1)^{n_1+(i-1)(n_2-1)+n_2(|a_1|+\dots+|a_{i-1}|)} \mathfrak m_{n_1}(a_1,\dots, \mathfrak m_{n_2}(a_i,\dots, a_{i+n_2-1}),\dots, a_n)
\end{multline*}
Using that $n_1=n-n_2+1$, as well as writing $d_A=\mathfrak m_1$, and bringing all terms to the left, we obtain
\begin{multline*}
\quad\quad \mathfrak m_1(\mathfrak m_n(a_1,\dots, a_n))+\sum_i (-1)^{(|a_1|+\dots+|a_{i-1}|)+n-1} \mathfrak m_n(a_1,\dots, \mathfrak m_1(a_i), \dots, a_n)
\\
+\sum_{n_1+n_2=n+1}\sum_i (-1)^{n_2(|a_1|+\dots+|a_{i-1}|)+(i-1)(n_2-1)+n-n_2} \mathfrak m_{n_1}(a_1,\dots, \mathfrak m_{n_2}(a_i,\dots, a_{i+n_2-1}),\dots, a_n)=0
\end{multline*}
These are precisely the signs for an $A_\infty$-algebra as stated in \cite[Proposition 1.4]{T}; see also \cite[Example 3.132]{MSS}. \item
For $\alpha=(\oo\oo\ii)$ there are only two cases for $\alpha_1$ and $\alpha_2$ (see Equation \eqref{EQN:boundary-(2,1)}). Both cases have $\alpha_1=(\oo\ii\ii)$ and $\alpha_2=(\oo\oo)$ ($m_1=1$, $n_1=2$, $m_2=2$, $n_2=0$), with either $i=1$, $j=2$, or with $i=2$, $j=1$. We get (recall $\mathfrak h=\mathfrak v_{(\oo\oo\ii)}$ and $c=\mathfrak v_{(\oo\oo)}=\sum_i c'_i\otimes c''_i$):
\[
[d_A,\mathfrak h]=(\mathfrak m_2) {}_1\circ_2 c- (\mathfrak m_2) {}_2\circ_1 c
\]
or, applied to $a\in A$, 
\[
d_A(\mathfrak h(a)) + (-1)^d \cdot \mathfrak h (d_A(a)) =\sum_i c'_i\otimes \mathfrak m_2(c''_i,a) -\sum_i (-1)^{|a|\cdot (|c'_i|+|c''_i|)} \mathfrak m_2(a,c'_i)\otimes c''_i.
\]
\item
In the case $\alpha=(\oo\oo\oo)$, we get three possible compositions of $\alpha_1$ and $\alpha_2$ that give $\alpha$, all of which with $\alpha_1=(\oo\oo\ii)$ and $\alpha_2=(\oo\oo)$ ($m_1=2$, $n_1=1$, $m_2=2$, $n_2=0$) and $i=j=1$. These compositions appear in the boundary of $\mathfrak l=\mathfrak v_{(\oo\oo\oo)}\in A^{\otimes 3}$ as
\[
 d_A (\mathfrak l)=(-1)^d \mathfrak h {\:}_1{\!}\circ_1 c+(-1)^d \pi. (\mathfrak h {\:}_1{\!}\circ_1 c)+(-1)^d \pi^2.(\mathfrak h {\:}_1{\!}\circ_1 c),
\]
where $\pi=(123)\in \Ss_3$; see \cite[Figure 25]{TZ} or \cite[Figure 15]{PT}. Using the components for $c=\sum_i c'_i\otimes c''_i$ as well as $\mathfrak h(c'_i)\otimes c''_i=\sum_j (c'_i)'_j\otimes (c'_i)''_j\otimes c''_i$, this may also be written as
\begin{multline*}
\quad\quad d_A(\mathfrak l)=
(-1)^d\cdot \sum_i \sum_j \Big((c'_i)'_j\otimes (c'_i)''_j\otimes c''_i
\\
+(-1)^{|c''_i|(|(c'_i)'_j|+|(c'_i)''_j|)}c''_i\otimes (c'_i)'_j\otimes (c'_i)''_j+ (-1)^{|(c'_i)'_j|(|(c'_i)''_j|+|c''_i|)}(c'_i)''_j\otimes c''_i\otimes (c'_i)'_j\Big).
\end{multline*}
\end{enumerate}
\end{ex}
The next lemma compares the data from Definition \ref{DEF:V-infty-algebra} to the notion of a $V_\infty$-algebra from \cite[Definition 3.1]{TZ}.
\begin{lem}\label{LEM:V-infty-algebra-comparison}
If $d$ is even, then the data of a $\mc V^{(d)}_\infty$-algebra is the same as that of a $V_\infty$-algebra in the sense of \cite[Definition 3.1]{TZ}. If $d$ is odd, then the two concepts differ at most by the signs in the boundary condition \eqref{EQN:dA(v-alpha)}.
\end{lem}
\begin{proof}
In \cite[Definition 3.1]{TZ}, the differential of $A$ was $v_1\in A\otimes A^*$, which is of degree $-1$, so that the choice made in \cite{TZ} is contrary to the one stated in Convention \ref{CONVENTIONS} where the differential has degree $+1$. Elements $v_{i_1,\dots,i_k}\in A\otimes (A^*)^{\otimes i_1}\otimes\dots\otimes A\otimes (A^*)^{\otimes i_k}$ in \cite{TZ} are of degree $k(2-d)+(d-4)+(i_1+\dots+i_k)$, which is equal to $-(4-(i_1+\dots +i_k)-2k+d(k-1))$, where $k$ is the number of ``$A$'' tensor factors (outgoing), and $(i_1+\dots +i_k)$ is the number of ``$A^*$'' tensor factors (incoming). Thus, switching the grading  of $A$ in \cite[Definition 3.1]{TZ} to its negative (i.e. setting degree $i$ to be $A^{-i}$), matches with our Convention \ref{CONVENTIONS} as well as the degree condition \eqref{EQN:|v-alpha|} for a $\mc V^{(d)}_\infty$-algebra $A$. Furthermore, the symmetry condition here matches the one in \cite{TZ}.

For the boundary condition \eqref{EQN:dA(v-alpha)}, note that, up to signs, the term appearing in the sum \eqref{EQN:dA(v-alpha)} are all choices of two allowable vertex terms whose composition gives the cyclic order of the $\alpha$ on the left. These are precisely the terms appearing in the boundary condition of \cite[Definition 3.1]{TZ}.

To check that the signs in \eqref{EQN:dA(v-alpha)} and \cite[Definition 3.1]{TZ} coincide, we now assume that $d$ is even. The signs of \eqref{EQN:dA(v-alpha)} were calculated on a summand in Equation \eqref{EQU:signs-in-dA}, which for even $d$ reduces to $d_A(\mathfrak v_\alpha)=\dots +(-1)^{n_1+(i-1)(n_2-1)+(m_1-1)(m_2-1)} (\pi,\sigma).\mathfrak v_{\alpha_1}{}_i\circ_j \mathfrak v_{\alpha_2}+\dots$. We can eliminate the ``$(m_1-1)(m_2-1)$''-part of the sign by redefining $\widetilde {\mathfrak v_\alpha}:=(-1)^{\frac{m(m-1)}{2}}\mathfrak v_\alpha$ for any $\alpha$ with $m$ outputs. Using this, the $\widetilde {\mathfrak v_\alpha}$ satisfy
\begin{equation}\label{EQU:sign-in-v-tilde}
d_A(\widetilde {\mathfrak v_\alpha})=\dots +(-1)^{n_1+(i-1)(n_2-1)} (\pi,\sigma).\widetilde {\mathfrak v_{\alpha_1}}{}_i\circ_j \widetilde {\mathfrak v_{\alpha_2}}+\dots,
\end{equation}
since for $m=m_1+m_2-1$ we have $\frac{m(m-1)}{2}=(m_1-1)(m_2-1)+\frac{m_1 (m_1-1)}{2}+\frac{m_2 (m_2-1)}{2}$. The remaining part of the signs come from shifts of the input tensor factors of $\widetilde {\mathfrak v_\alpha}\in Hom(A^{\otimes n},A^{\otimes m})\cong A^{\otimes m}\otimes (A^*)^{\otimes n}$ as stated in \cite[Definition 3.1]{TZ}. In fact, for the map $s:A[1]\to A$, $s(a):=a$ which shifts up by $1$, we define $\widetilde{\widetilde {\mathfrak v_\alpha}}:=\widetilde {\mathfrak v_\alpha}\circ (s\otimes \dots\otimes s)\in Hom((A[1])^{\otimes n},A^{\otimes m})$, where we note that the degree of $s$ is $|s|=1$, and $|\widetilde {\mathfrak v_\alpha}|=4-n-2m+d(m-1)\equiv n \,\,(mod\,\, 2)$. To keep track of the tensor factors to which $s$ applies, we will use the notation $\widetilde{\widetilde {\mathfrak v_\alpha}}=\widetilde {\mathfrak v_\alpha}\circ (s^{\{1\}}\otimes \dots\otimes s^{\{n\}})$, where $s^{\{i\}}$ applies to the $i$th $A[1]$-tensor factor. We claim that these shifted $\widetilde{\widetilde {\mathfrak v_\alpha}}$s now compose as defined in \cite[Definition 3.1]{TZ}. More precisely, the composition $\widetilde{\widetilde {\mathfrak v_{\alpha_1}}}{}_i\circ_j \widetilde{\widetilde {\mathfrak v_{\alpha_2}}} = \Big(\widetilde {\mathfrak v_{\alpha_1}} \circ (s_1^{\{1\}}\otimes \dots\otimes s_1^{\{n_1\}})\Big){}_i\circ_j \Big(\widetilde {\mathfrak v_{\alpha_2}} \circ (s_2^{\{1\}}\otimes \dots\otimes s_2^{\{n_2\}})\Big) $ is given by first removing the shift $s_1^{\{i\}}$ (which would receive an output of $\widetilde {\mathfrak v_{\alpha_2}}$ in $A$ instead of $A[1]$ and thus needs no shift, and which costs a sign of $(-1)^{n_1+i-1}$ when removing $s_1^{\{i\}}$ from the left), and then rearranging the shifts in the order of the inputs. This yields
\begin{equation*}
\widetilde{\widetilde {\mathfrak v_{\alpha_1}}}{}_i\circ_j \widetilde{\widetilde {\mathfrak v_{\alpha_2}}}= (-1)^{\epsilon} \cdot (\widetilde {\mathfrak v_{\alpha_1}} {}_i\circ_j \widetilde {\mathfrak v_{\alpha_2}})\circ (s_1^{\{1\}}\otimes \dots\otimes s_1^{\{i-1\}}\otimes s_2^{\{1\}}\otimes \dots\otimes s_2^{\{n_2\}}\otimes s_1^{\{i+1\}}\otimes \dots\otimes s_1^{\{n_1\}}),
\end{equation*}
where $\epsilon=n_1+i-1+(i-1)n_2$, since $s_1^{\{1\}}\otimes \dots\otimes s_1^{\{i-1\}}$ moved over $\widetilde {\mathfrak v_{\alpha_2}}$, while $s_1^{\{i+1\}}\otimes \dots\otimes s_1^{\{n_1\}}$ moved over all of $ \widetilde{\widetilde {\mathfrak v_{\alpha_2}}}$ of even degree. Comparing this to \eqref{EQU:sign-in-v-tilde}, we see that in this shifted setting, $d_A(\widetilde{\widetilde {\mathfrak v_\alpha}})=\dots +(\pi,\sigma).\widetilde{\widetilde {\mathfrak v_{\alpha_1}}}{}_i\circ_j \widetilde{\widetilde {\mathfrak v_{\alpha_2}}}+\dots$ without applying any further signs. This is just as the $v_{i_1,\dots i_n}$ compose in \cite[Definition 3.1]{TZ}, i.e. the only additional signs come from the Koszul sign rule. This shows that the $\widetilde{\widetilde {\mathfrak v_{\alpha}}}$ satisfy the conditions from \cite[Definition 3.1]{TZ} including the signs.
\end{proof}
We close this section by recalling the main theorem from \cite{TZ}, which states that $V_\infty$-algebras permit a string topology-type action (cf. \cite{CS}) of a graph complex $\mc{DG}^\bullet_\infty$ on the cyclic Hochschild complex $CC^\bullet(A)$ of $A$. The last lemma thus implies the following corollary.
\begin{cor}[\cite{TZ}, Theorem 4.3]\label{COR:DG-action}
If $d$ is even, and $A$ is a $\mc V^{(d)}_\infty$ algebra, then there is a map of PROPs $\mc{DG}^\bullet_\infty\to \mc End(CC^\bullet(A))$.
\end{cor}
We refer to \cite{TZ} for further details of this action. 

\section{Structures related to $\mc V^{(d)}$}\label{SEC:related-structures}

In this section we consider two concepts closely related to the dioperad $\mc V^{(d)}$. First, in section \ref{SEC:V-properad}, we look at the corresponding properad $\mc F_{dioperad}^{properad}(\mc V^{(d)})$ under the map described in \cite[Section 5.6.]{MV}, and we show that this properad is not contractible Koszul. Second, in section \ref{SEC:W-dioperad}, we consider a version of $\mc V^{(d)}$ with anti-symmetric co-inner product.

\subsection{$\mc V^{(d)}$ as a properad}\label{SEC:V-properad}

In this section, we freely use the notation defined by Merkulov and Vallette in \cite[Section 5.6.]{MV} to compare dioperadic and properadic structures. If $\mc U_{dioperad}^{properad}:\text{Properads}\to \text{Dioperads}$ is the forgetful functor, then its left adjoint is denoted by $\mc F_{dioperad}^{properad}:\text{Dioperads}\to \text{Properads}$. By \cite[Corollary 45]{MV}, $\mc F_{dioperad}^{properad}(\mc V^{(d)})$ is the quadratic properad given by the (properad) generators $E_{\mc V^{(d)}}$ and relations $R_{\mc V^{(d)}}$ described in Definition \ref{DEF:V-dioperad0}. We define the properad $\mc {PV}^{(d)}:=\mc F_{dioperad}^{properad}(\mc V^{(d)})$. We will show that $\mc {PV}^{(d)}$ is not Koszul contractible.

Recall from \cite[Proposition 48]{MV}, that $\mc {PV}^{(d)}$ is Koszul contractible iff the resolution $\bD\mc V^{(-d)}\to \mc V^{(d)}$ induces a quasi-isomorphism when applying $\mc F_{dioperad}^{properad}$, i.e. iff the map $\mc F_{dioperad}^{properad}(\bD\mc V^{(-d)})\to \mc F_{dioperad}^{properad}(\mc V^{(d)})=\mc {PV}^{(d)}$ is a quasi-isomorphism.
\begin{prop}\label{PROP:V-not-Koszul}
The properad $\mc {PV}^{(d)}$ is not Koszul contractible.
\end{prop}
The proof is given below. Recall from \cite[Section 5.6.]{MV} that a Koszul contractible properad $\mc P$ is also Koszul as a properad, however the converse is not true. (In fact, by \cite[Remark below Corollary 51]{MV}, $\mc Frob_\diamond$ is Koszul but not Koszul contractible. On the other hand, by \cite[Text below Proposition 48, and Theorem 53]{MV}, $\mc{NC}\text{-}\mc Frob$ is neither Koszul nor Koszul contractible, even though its associated dioperad is Koszul as a dioperad.) We can thus still ask the following question.
\begin{question}
Is $\mc {PV}^{(d)}$ Koszul as a properad?
\end{question}

\begin{proof}[Proof of Proposition \ref{PROP:V-not-Koszul}]
Since $\bD \mc V^{(-d)}$ is a free dioperad generated by $(\mc {V}^{(-d)})^*$ (ignoring shifts and the extra differential $d^\circ$), it follows from \cite[Corollary 45]{MV}, that $\mc F_{dioperad}^{properad}(\bD\mc V^{(-d)})$ is also a free properad on the same generators (up to shift and extra differential). If $\mc {PV}^{(d)}$ were Koszul contractible, then $\mc F_{dioperad}^{properad}(\bD\mc V^{(-d)})$ cannot have any ``higher'' homology. We show this is false by explicitly computing $(\mc F_{dioperad}^{properad}(\bD\mc V^{(-d)}))^{\text{[}genus=1\text{]}}(2,0)$, which is the genus $1$ part of the space with $0$ inputs and $2$ outputs of $\mc F_{dioperad}^{properad}(\bD\mc V^{(-d)})$. In fact, we show that $(\mc F_{dioperad}^{properad}(\bD\mc V^{(-d)}))^{\text{[}genus=1\text{]}}(2,0)$ is isomorphic to the cellular chains of a disjoint union of a cylinder and a disk (given by a certain cellular decomposition). The complete space $(\mc F_{dioperad}^{properad}(\bD\mc V^{(-d)}))^{\text{[}genus=1\text{]}}(2,0)$ is displayed in Figure \ref{FIG:FDV1(2,0)}. In particular, the cylinder induces a non-trivial ``$1$''-cycle $\xi_1+\xi_2$, which induces a ``higher'' homology class.
\begin{figure}[h]
\[
\includegraphics[scale=1]{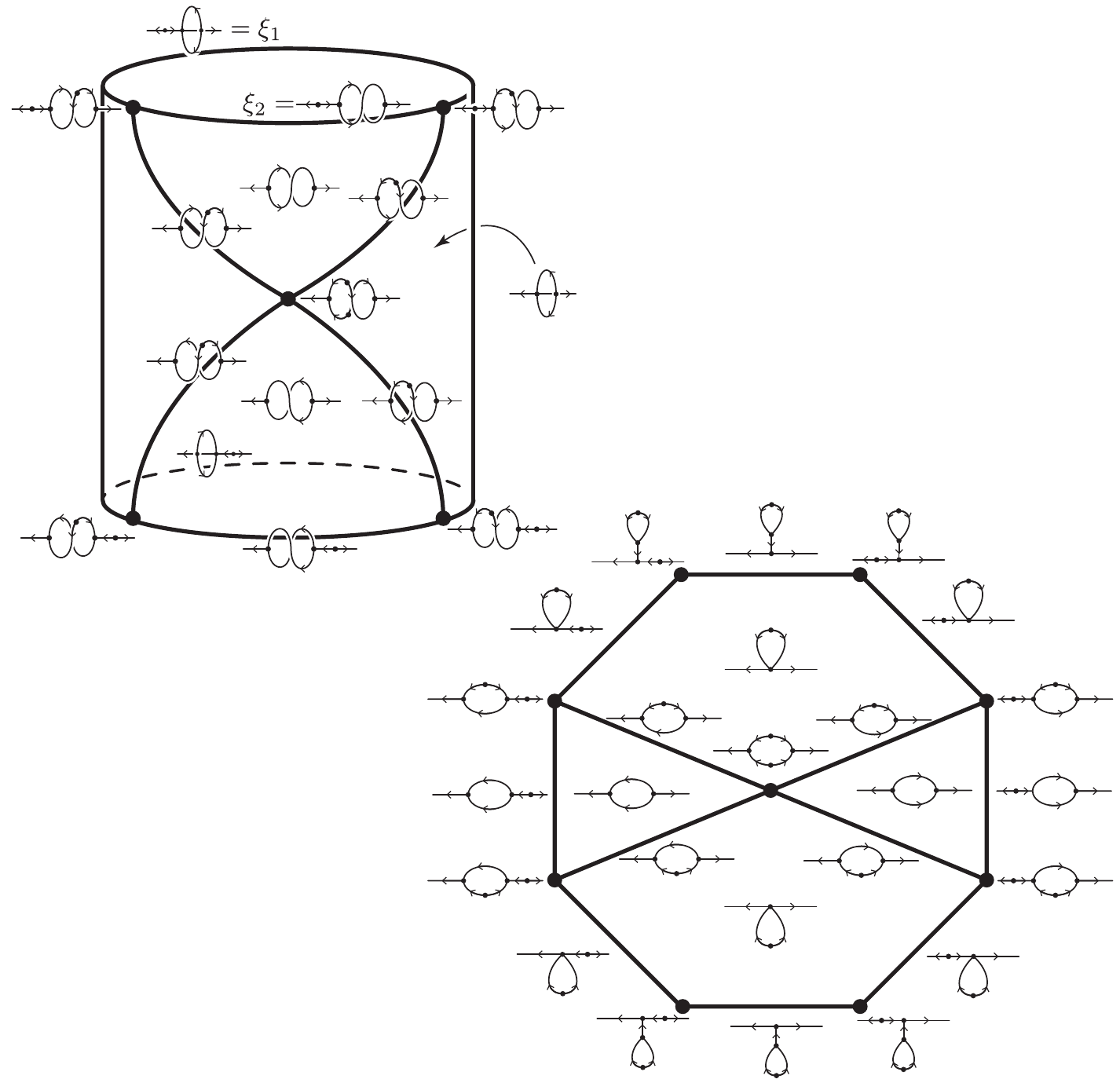}
\]
\caption{The space $(\mc F_{dioperad}^{properad}(\bD\mc V^{(-d)}))^{\text{[}genus=1\text{]}}(2,0)$ is given as a union of a decomposed cylinder (union of two triangles and one hegaxon; see upper left) with a decomposed disk (union of two triangles and two pentagons; see lower right). (The labelings of each directed graph are suppressed for better readability; they are always $1$ on the left root and $2$ on the right root of a graph.) The non-trivial homology class is represented, for example, by the cycle $\xi_1+\xi_2$.}\label{FIG:FDV1(2,0)}
\end{figure}
\end{proof}

\subsection{Antisymmetric co-inner products}\label{SEC:W-dioperad}

Recall our main example of a  $\mc V^{(d)}$-algebra being the cohomology $H^\bullet(X)$ of an even dimensional, oriented manifold $X$ with its usual cup product and with the image of the Thom class $c=\sum_i c'_i\otimes c''_i\in H^\bullet(X)\otimes H^\bullet(X)$. If $X$ is an odd dimensional manifold, then the Thom class $c$ is anti-symmetric,\[\sum_i c'_{i}\otimes c''_{i}= (-1)\cdot \sum_i (-1)^{|c'_i||c''_i|}\cdot c''_i\otimes c'_i.\]
Thus $H^\bullet(X)$ with its cup product and Thom class image is an algebra over the dioperad $\mc W^{(d)}$ which is an antisymmetric version of $\mc V^{(d)}$, defined below.

Explicitly, we define $\mc W^{(d)}=\langle E_{\mc W^{(d)}},R_{\mc W^{(d)}}\rangle$ to be generated by $E_{\mc W^{(d)}}(1,2)=\kk\cdot \mu\oplus \kk\cdot \bar \mu$ with right $\Ss_2$-action $\sigma.\mu=\bar \mu$ interchanging $\mu$ and $\bar \mu$, just as in $\mc V^{(d)}$ (see Definition \ref{DEF:V-dioperad0}), and $E_{\mc W^{(d)}}(2,0)=\kk[-d]\cdot \nu$ be concentrated in degree $d$ with left $\Ss_2$-action $\sigma.\nu=-\nu$. The space of relations is again spanned by associativity of $\mu$ and invariance of $\nu$, i.e. $\mu _1\circ_1\mu=\mu _2\circ_1\mu$, and $\mu _1\circ_2\nu=\mu _2\circ_1 \nu$ (see Figure \ref{FIG:V-relations}).

It would be interesting to find explicit expressions for $\mc W^{(d)}(m,n)$ and prove its Koszulness, just as we did for $\mc V^{(d)}$. Note, that there are elements $\gamma_{\llbracket\alpha,\ooo,\iii\rrbracket}\in \mc W^{(d)}(m,n)$ which span all of $\mc W^{(d)}(m,n)$ via the same proof as in Proposition \ref{PROP:V(m,n)-as-gammas} \eqref{ITM:V(m,n)=gammas}. However, due to sign reasons some of the $\gamma_{\llbracket\alpha,\ooo,\iii\rrbracket}$ might be zero, as for example:
\[
\includegraphics[scale=.85]{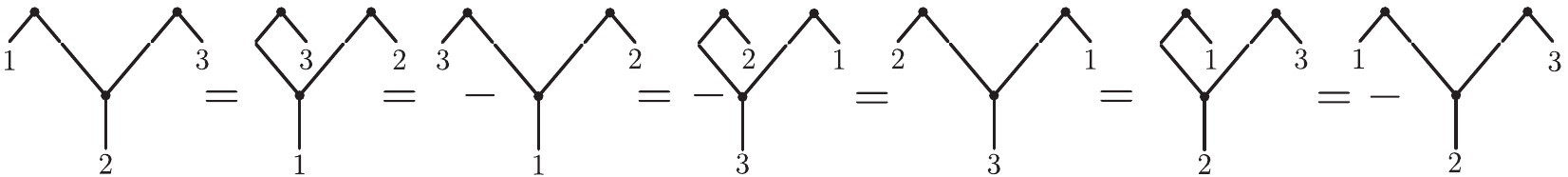}
\]
Thus, for $2\neq 0$ we get $\gamma_{(\oo\oo\oo)}=0$, and so $\mc W^{(d)}(3,0)=\{0\}$. In general, 
the chain complex
$\mc W^{(d)}(m,n)$ is a quotient of $\mc V^{(d)}(m,n)$, where some $\gamma_{\llbracket\alpha,\ooo,\iii\rrbracket}$ are set equal to zero.

\begin{question}
Explicitly describe $\mc W^{(d)}(m,n)$ combinatorially.
\end{question}
\begin{question}\label{Q:W-Koszul-properad}
Is $\mc W^{(d)}$ Koszul as a dioperad?
\end{question}

We note that for our main example $H^\bullet(X)$ of an odd dimensional, orientable manifold $X$, the composition of the co-inner product with the product is zero, since $\sum_i c'_i\cdot c''_i$ is the Euler class, which vanishes for odd dimensional manifolds. Thus, $H^\bullet(X)$ is an algebra over the properad $\mc F_{dioperad}^{properad}(\mc W^{(d)})_\diamond$ with additional diamond relation $\mu \circ \nu=0$.

{
\bibliographystyle{amsalpha}
\bibliography{VinftyKoszuality-arXiv}}

\end{document}